\documentclass[[a4paper,english,12pt]{amsart}
\usepackage[utf8]{inputenc}
\usepackage{amsfonts}
\usepackage{amsmath,amsthm,amssymb}
\usepackage{bbm}
\usepackage{verbatim}
\usepackage{mathabx}
\usepackage{xcolor}
\usepackage{comment}

\pdfpagewidth\paperwidth
\pdfpageheight\paperheight
\theoremstyle{plain}

\newtheorem{theorem}{Theorem}

\newtheorem{lemma}[theorem]{Lemma}
\newtheorem{corollary}[theorem]{Corollary}
\newtheorem{definition}[theorem]{Definition}

\newtheorem*{assumption-o}{Assumption O}

\theoremstyle{remark}

\newcommand\supp{\operatorname{supp}}

\usepackage{color}
\usepackage[normalem]{ulem}

\definecolor{blue(ryb)}{rgb}{0.01, 0.28, 1.0}
\definecolor{blue-green}{rgb}{0.00, 0.67, 0.9}

\definecolor{darkgreen}{rgb}{0,0.6,0.2}

\newcommand{\intt}{\mbox{\rm int\,}}

\newcommand{\conv}{\mathop{\mathrm{conv}}}

\newcommand{\ff}{\varphi}

\newcommand{\Om}{{\Omega}}

\newcommand{\ve}{\varepsilon}

\newcommand{\RR}{\mathbb R}
\newcommand{\NN}{\mathbb N}
\newcommand{\CC}{\mathbb C}

\DeclareMathOperator*{\esssup}{ess\,sup}

\usepackage{babel}

\begin{document}

%%%%%%%%%%%%%%%%%%%%%%%%%%%%%%%%%%%%%%%%%%%%%%%%%%%%%%%%%
%%%%%%%%%%%%%%%%%%%%%%%%%%%%%%%%%%%%%%%%%%%%%%%%%%%%%%%%

\title{Duality for Delsarte's extremal problem on compact Gelfand pairs}

\author[Berdysheva at al.]{Elena E. Berdysheva, B\'alint Farkas, Marcell Ga\'al, Mita D. Ramabulana, Szil\'ard Gy. R\'ev\'esz}

%\dedicatory{Dedicated to the memory of ...}

\address{Elena E. Berdysheva
\newline  \indent University of Cape Town,
\newline  \indent South Africa}

\email{elena.berdysheva@uct.ac.za}

\address{B\'alint Farkas
\newline  \indent University of Wuppertal
\newline  \indent Gau{\ss}stra{\ss}e 20, 42119 Wuppertal, Germany}

\email{farkas@math.uni-wuppertal.de}

\address{Mita D. Ramabulana
\newline  \indent University of Cape Town,
\newline  \indent South Africa}

\email{mita.ramabulana@uct.ac.za}

\address{Szil\'ard Gy. R\'evesz
\newline  \indent HUN-REN R\'enyi Institute of Mathematics,
\newline \indent Budapest, Re\'altanoda utca 13-15,1053 HUNGARY}

\email{revesz.szilard@renyi.hu}


\maketitle

\begin{abstract}
 We study Delsarte-type problems for positive definite functions on compact Gelfand pairs as infinite-dimensional linear programming problems. This setup includes, as a particular case, the case of compact Abelian groups. Depending on the restriction on the signs of the functions, we obtain two important particular cases, the Tur\'an and Delsarte problems. These problems have been studied in relation to number theory, sphere packing, and statistics. In this paper, we describe their duals and prove a strong duality statement. The  method also allows for more general normalization constraints than the prominent one for the Delsarte and the Tur\'an problems, i.e., the one concerning point evaluation at $0$.
\end{abstract}

\section{Introduction} \label{section:introduction}

Extremal problems for positive definite functions play an important role in harmonic analysis, number theory, and geometry.  A common theme in this topic is to study these problems as linear programming problems, describe their duals, and seek a strong duality statement. The approach is often called ``Delsarte's method'', following the work in coding theory of Delsarte in~\cite{delsarte1}, and the work of Delsarte, Goethals, and Seidel in~\cite{delsarte2}.

In Ruzsa~\cite{Ruzsa1981, Ruzsa1982}, arithmetic properties of sequences like uniform distribution mod~1, intersectivity, van der Corput sequences and alike were analyzed by means of Delsarte type  extremal problems and their duals on $\mathbb{Z}$. In R\'ev\'esz~\cite{revesz, Revesz-Beurling, R-Landau} and Virosztek~\cite{virosztek}, duality theory of various Delsarte-type problems on $\mathbb{Z}^d$ was studied with the aim of a direct application to the Beurling theory of primes. The admissible functions of these extremal problems are often taken with prescriptions on where they can be non-zero, non-negative, or non-positive. It seems that these papers were the first to deal with two sign conditions assumed simultaneously, i.e., under the assumption that $a(m)\le 0$ for $m \in  M \subset \mathbb{Z}^d$, whereas $a(k)\ge 0$ for $k \in L \subset \mathbb{Z}^d$, say.  In this direction, the dual of a certain extremal problem of Landau was found in R\'ev\'esz~\cite{R-Landau}. Variants of Landau's extremal problem are still used in establishing zero-free regions for arithmetical zeta functions, see, e.g., \cite{frederik}. Duality in the case of finite Abelian groups was studied by Matolcsi, Ruzsa~\cite{MatolcsiRuzsa}. 

A notable example is the Delsarte problem in $\mathbb{R}^d$, which provides bounds for sphere packing densities. In this setting, the problem is as follows. Given a centrally symmetric convex body $\Omega \subset \mathbb{R}^d$, find 
\begin{equation} \label{eq:first_delsarte}
    \sup\int_{\mathbb{R}^d} \ff(x) \,\mbox{d}x,
\end{equation}
where $\ff$ is a continuous positive definite function satisfying $\ff(0)=1$ as well as ${\ff|_{\mathbb{R}^d\setminus\Omega} \le 0}$. Delsarte-type problems in the context of sphere packing were studied in Kabatyanskii, Levenshtein~\cite{KabLev}, Levenshtein~\cite{Levenshtein}; and precisely problem~\eqref{eq:first_delsarte} with slight variations in the function class was investigated for $\Omega$ being a ball in~$\mathbb{R}^d$ in Yudin~\cite{Yudin}, Gorbachev~\cite{GorbachevTula, gorbachev}, Cohn, Elkies~\cite{cohn-elkies}, Cohn~\cite{Cohn},   Viazov\-ska~\cite{viazovska}, Cohn, Kumar, Miller, Radchenko, Viazovska~\cite{CKMRV}, Cohn, de Laat, Salmon~\cite{Cohn-Laat-Salmon}. Strong duality theory for this Delsarte problem in $\mathbb{R}^d$ was solved in Cohn, de Laat, Salmon~\cite{Cohn-Laat-Salmon}, Kolountzakis, Lev, Matolcsi~\cite{KLM}.

Another notable problem is the so-called Tur\'an extremal problem which is concerned with finding~\eqref{eq:first_delsarte},
where $\ff$ is a continuous positive definite function supported in a centrally symmetric convex body $\Omega$, and satisfying $\ff(0)=1$. The Tur\'an problem for~$\Omega$ being a ball in~$\mathbb{R}^d$ was introduced in Siegel~\cite{siegel} in relation to Minkowski's Lattice Point Theorem. It has been investigated for both the ball and convex tiles in~$\mathbb{R}^d$ in Arestov, Berdysheva~\cite{areberd,areberd0}, Gorbachev~\cite{gorbachev}, Kolountzakis, R\'ev\'esz~\cite{kolrev}.
The above extremal problems make sense in LCA groups, see  R\'ev\'esz~\cite{reveszLCA}, Berdysheva, R\'ev\'esz~\cite{Berdysheva}, Ga\'al, Nagy-Csiha~\cite{marcell-zsuzsa}, Ramabulana~\cite{ramabulanaexistence}, Berdysheva, Ramabulana, R\'ev\'esz~\cite{BRR_Expo}. In \cite{BRR_Expo} the question of the existence of an extremal function as well the influence of certain topological conditions on the sign or support restriction sets has been analysed. The question of the existence of an extremal function has also been studied in Cohn, de Laat, Salmon~\cite{Cohn-Laat-Salmon}, Kolountzakis, Lev, Matolcsi~\cite{KLM}, Ramabulana~\cite{ramabulanaexistence}. Duality for the Tur\'an problem with $\Omega$ being a ball in $\mathbb{R}^d$ was studied by Gabardo \cite{gabardo}.

In this paper, we consider the problems for compact Gelfand pairs. Part of our motivation comes from the spherical Tur\'an problem for positive definite kernels introduced in Gneiting~\cite{gneiting} in connection with some problems in geophysical, meteorological, and climatological modelling. In particular, for $0 < c \le \pi$, the spherical Tur\'an problem asks for 
\begin{equation}\label{spherturan}
    \mathcal{T}_{\mathbb{S}^d}(c):=\sup_{\psi}\int_{\mathbb{S}^d}\psi(\theta(x,y))\textup{d}y,
\end{equation}
where $x \in \mathbb{S}^d$ is an arbitrary basepoint, $\psi: [0, \pi] \to \mathbb{R}$ is a continuous function with $\psi(0)=1$, $\psi(t)=0$ for $t \ge c$, and such that the isotropic (i.e., its values at $(x,y) \in \mathbb{S}^d$ depend only on the distance between $x$ and $y$) function $\psi \circ \theta:\mathbb{S}^d \times \mathbb{S}^d \to \mathbb{R}$ is positive definite. The integral in \eqref{spherturan} is with respect to the surface measure on the sphere $\mathbb{S}^d$.

In applications, for $d=2$, the sphere \(\mathbb{S}^2\) models the surface of the Earth. The positive definite kernels occur as covariances of real-valued random variables on the sphere. In particular, if \( \Phi(x, y) = \text{Cov}(Z(x), Z(y)) \) is the covariance of a real-valued random field \( Z(x) \) on \( \mathbb{S}^d \), then the quantity
\[
\sup_\Phi \int_{\mathbb{S}^d} \Phi(x, y) \, \mbox{d}y
\]
is the maximum total covariance that can be concentrated locally around \( x \), without violating the constraint of positive definiteness, and while ensuring zero correlation beyond radius \( c \). For more details and applications see, e.g., Gaspari and Cohn~\cite{Gaspari-Cohn}, and Hamill, Whitaker and Snyder~\cite{Hamill}\footnote{We thank Tilmann Gneiting for references and discussions on statistical applications.}.

In the paper~\cite{ramabulana} by Ramabulana, problem~\eqref{spherturan} was generalised to homogeneous spaces and shown to be equivalent to a Tur\'an problem on the full symmetry group of the homogeneous space. Note that for the sphere \(\mathbb{S}^d\), this amounts to considering the Tur\'an problem on the compact Gelfand pair \((SO(d+~1),SO(d))\). Any compact Abelian group \(G\) forms a compact Gelfand pair \((G, \{0\})\), so that this is also a generalisation of the classical Tur\'an problem considered on the compact Abelian group \(\mathbb{T}^{d}\). So, the setup in this paper aims to cover these cases. We should also mention that the Delsarte problem is of interest in Gelfand pairs other than \((SO(d+1), SO(d))\), see Cohn,  Zhao \cite{cohn-zhao}, Wackenhuth~\cite{wackenhuth} and the references within. Note that the Gelfand pair case is related to spherical codes, kissing numbers, and other packing problems on other homogeneous spaces like hyperbolic space, Cohn,  Zhao~\cite{cohn-zhao},  Wackenhuth~\cite{maximilian}, Kuklin~\cite{kuklin}, Arestov, Babenko~\cite{AB}. In this direction, the duality theory has been considered in relation to the kissing number problem in Arestov, Babenko~\cite{AB}. In Gorbachev~\cite{gorbachevduality}, the duality was considered for Delsarte problem on compact homogeneous Riemannian spaces of rank~1 in connection with weighted designs; note that in \cite{gorbachevduality} the sign conditions are on the Fourier coefficients, while in the present manuscript they are on the functions.  The Delsarte scheme in non-commutative compact groups was also used in relation to the problem of mutually unbiased bases by Kolountzakis, Matolcsi, and Weiner in~\cite{kolmatwei}, and by Matolcsi and Weiner in~\cite{MW}. 

Motivated by R\'ev\'esz~\cite{revesz}, we study these problems as infinite-dimensional linear programming problems. We follow the approach of Arestov and Babenko in~\cite{AB}. Arestov and Babenko studied the Delsarte problem in a particular setting (for series in terms of ultraspherical polynomials on an interval). They formulated the dual problem and established the strong duality; this helped them to find the solution of their version of the Delsarte problem. In their study of duality, they used the theory from Gol'shtein's book~\cite{golstein}. We present the relevant theory of duality in Section~\ref{section:duality_Duffin}.

Our main result is that a strong duality relation holds for a general class of Delsarte-type problems on compact Gelfand pairs; see Theorems~\ref{theorem:main_duality_G} and~\ref{theorem:main_duality}. We note that the method of Arestov and Babenko in~\cite{AB} could only handle a one-sided sign restriction while our development of it can handle two-sided sign restrictions.  Also, we can consider more general normalization constraints than the point evaluation at~$0$. On the other hand, these methods can only handle the compact case. Another method, less general, e.g., regarding sign restrictions, but tackling the non-compact case at least in locally compact Abelian groups, is presented in our companion paper \cite{C-paper}. Let us formulate our main result here for compact Abelian groups.

 Let $G$ be a compact Abelian group with identity $0$, normalised Haar measure $\lambda_G$, and dual group $\widehat{G}$. Let $C(G)$ denote the collection of continuous real-valued functions, and $\widehat{\varphi}$ denote the Fourier transform of $\varphi \in C(G)$. By $M(G)$ we denote the collection of real-valued regular signed Borel measures on $G$. Denote by $\delta_0 \in M(G)$ the Dirac measure at the identity $0 \in G$.

 A measure $\mu \in M(G)$ is called positive definite (or of positive type) if for all continuous ``test functions'' $u\in C(G)$ the integral of $\mu$ against the convolution square $u \star \widetilde{u}$ is non-negative: $\int_G u \star~ \widetilde{u} \,\mbox{d} \mu \ge~ 0$, where $\widetilde{u}(g) := \overline{u(-g)}$.
 
Let $\Omega_+$, $\Omega_-$ be symmetric Borel subsets of $G$ with $0 \in \intt{\Omega_+}$ and let $\Omega_{\pm}^c$ denote their complements in $G$. Consider the function class
$$
\mathcal{F}_G(\Omega_+, \Omega_-) := \{ \varphi \in C(G) : \widehat{\varphi} \ge 0, \ \varphi(0) = 1, \  \varphi|_{\Omega_+^c} \le 0, \  \varphi|_{\Omega_-^c} \ge 0 \}.
$$

We will impose the following topological condition. 

\begin{assumption-o}

We say that a set $S \subset G$ satisfies \emph{Assumption O}, if for any $g \in S$, any open neighbourhood $V$ of $g$ has $ \lambda_G(V \cap S) > 0$.

\end{assumption-o}

We will put a condition that the sets $\Omega_+^c$ and $\Omega_-^c$ satisfy Assumption O. This is obviously the case when the sets $\Omega_+$, $\Omega_-$ are closed (and the sets $\Omega_+^c$, $\Omega_-^c$ are open). The condition is also satisfied if $\Omega_+$, $\Omega_-$ (and therefore also $\Omega_+^c$, $\Omega_-^c$) are boundary-coherent in the terminology of Berdysheva, Ramabulana, R\'ev\'esz~\cite{BRR_Expo}. Also the condition that $\Omega_+$, $\Omega_-$ have continuous boundary, which is considered in Kolountzakis, Lev, Matolcsi~\cite{KLM} in the case of $G = \RR^d$, is stronger than the boundary-coherence and guarantees Assumption~O.

Let us write for any subset $\Theta \subset G$ 
$$
M^*(\Theta):= \{ \mu \in M(G): \mu\text{ is a non-negative measure and
} \supp(\mu)\subseteq\overline{G\setminus\Theta} \}.
$$
We denote by $T_0(G)$ the subset of $M(G)$ consisting of all $\tau \in M(G)$ such that $\tau$ is positive definite and $\tau(G) = 0$.

\begin{theorem} \label{theorem:main_group}
Let $G$ be a compact Abelian group, and let $\Omega_+$, $\Omega_-$ be symmetric Borel subsets of $G$ with $0 \in \intt{\Omega_+}$, and $\Omega_+^c$, $\Omega_-^c$ satisfying Assumption~O. The linear programming problem 
\begin{equation*}
\mathcal{C}_G(\Omega_+, \Omega_-) := \sup \left\{ \int_G \varphi \,\textup{d}\lambda_G : \varphi \in \mathcal{F}_G(\Omega_+, \Omega_-) \right\}
\end{equation*}
has dual problem
\begin{align*}
\mathcal{C}_{G}(\Omega_+, \Omega_-)^{*}:=\inf \{ \alpha \in \RR : \alpha \delta_0 - \lambda_G \in  M^*(\Omega_-) - M^*(\Omega_+) + T_0(G)\},
\end{align*}
and we have the strong duality relation
\begin{equation*}
\mathcal{C}_G(\Omega_+, \Omega_-) = \mathcal{C}_G(\Omega_+, \Omega_-)^*.
\end{equation*}

\end{theorem}

The theorem takes a simpler form for the classical Delsarte problem of type~\eqref{eq:first_delsarte}.

\begin{theorem} \label{theorem:main_group_G}
Let $G$ be a compact Abelian group, and let $\Omega_+$ be a symmetric Borel subset of $G$ with $0 \in \intt{\Omega_+}$. The Delsarte problem 
\begin{equation*}
\mathcal{D}_G(\Omega_+) := \mathcal{C}_G(\Omega_+, G) = \sup \left\{ \int_G \varphi \,\textup{d}\lambda_G : \varphi \in \mathcal{F}_G(\Omega_+, G) \right\}
\end{equation*}
has dual problem
\begin{align*}
\mathcal{D}_G(\Omega_+)^* := \mathcal{C}_{G}(\Omega_+, G)^{*} = \inf \{ \alpha \in \RR : \alpha \delta_0 - \lambda_G \in - M^*(\Omega_+) + T_0(G)\},
\end{align*}
and we have the strong duality relation
\begin{equation*}
    \mathcal{D}_G(\Omega_+) = \mathcal{D}_G(\Omega_+)^*.
\end{equation*}

\end{theorem}

Theorems~\ref{theorem:main_group} and~\ref{theorem:main_group_G} follow from the more general Theorems~\ref{theorem:main_duality} and~\ref{theorem:main_duality_G}; see the remarks at the end of Section~\ref{section:A-B_scheme}.

\section{Notation and Preliminaries} \label{section:notation}

We will use the following notation. For a set $S \subset G$ we denote by $\intt{S}$ the interior of $S$, by $\overline{S}$ its closure, and by $S^c$ its complement. For two sets $N$, $S$ the notation $N \Subset S$ means that $N$ is a compact subset of $S$. We say that $S$ is symmetric if $S^{-1} = S$. For a subset \(S\subset G\), \(\mathbbm{1}_{S}\) is its indicator function satisfying \(\mathbbm{1}_S(g) = 1\) for all \(g \in S\) and \(\mathbbm{1}_ S(g) = 0\) for all \(g \notin S\).

Let $G$ be a compact group with identity \(e\) and $K$ a closed subgroup of $G$. Denote by $\lambda_{G}$ and $\lambda_K$ the normalised (i.e., $\lambda_{G}(G)=1=\lambda_{K}(K)$) Haar measures of $G$ and~$K$, respectively. A function $\ff: G \to \mathbb{R}$ is \emph{$K$-bi-invariant} if $\ff(kgk') = \ff(g)$ for all $k, k' \in  K$ and for all $g \in G$. We shall call a subset $U \subset G$ $K$-bi-invariant if its indicator function~$\mathbbm{1}_{U}$ is $K$-bi-invariant. For $1 \le p < \infty$, we denote by $L^{p}(G)$ the collection of real-valued $p$-integrable $\ff: G \to \mathbb{R}$ satisfying
\begin{equation*}
    \|\ff\|_{p}:= \left(\int_G|\ff(g)|^p \,\mbox{d}\lambda_G(g) \right)^\frac{1}{p}, \quad 1 \le p < \infty,
\end{equation*}
by $L^{\infty}(G)$ the collection of functions $\ff: G \to \mathbb{R}$ satisfying
\begin{equation*}
    \|\ff\|_\infty:= \esssup_{g \in G}|\ff(g)|,
\end{equation*}
and by $C(G)$ the collection of continuous functions $\ff: G \to \mathbb{R}$.
We denote by $L^{p}(G)^K$ the subset of $L^p(G)$ consisting of $K$-bi-invariant functions and $C(G)^{K}$ the subset of $C(G)$  consisting of $K$-bi-invariant functions. The convolution $\ff \star \psi$ of functions $\ff, \psi: G \to \mathbb{R}$ is defined as
\begin{equation*}
\ff \star \psi (g) := \int_{G} \ff(h) \psi(h^{-1}g) \,\mbox{d}\lambda_{G}(h), \quad g \in G,
\end{equation*}
whenever the integral exists.

The pair $(G,K)$ is called a (compact) \emph{Gelfand pair} if the convolution algebra $L^1(G)^{K}$ is commutative. For a function $\ff : G \to \RR$, we will write $\ff \gg 0$ if $\ff$ is \emph{positive definite}, i.e., if
\begin{equation*}\label{posdefdef}
\sum_{n=1}^{N} \sum_{m=1}^N c_n \overline{c_m} \ff(g_n^{-1}g_m) \ge 0
\end{equation*}
for all $N \in \NN$, $g_1, \ldots, g_N \in G$ and $c_1,\ldots, c_N \in \CC$. Denote by $M(G)$ the space of real-valued regular signed Borel measures on $G$, and by $M(G)^K$ its subset consisting of $K$-bi-invariant measures $\mu$, i.e., of measures $\mu$ satisfying $\mu(kAk')=\mu(A)$ for all $k, k' \in K$ and all Borel sets~$A$. For $\mu \in M(G)$, the $K$-periodisation of $\mu$ is the $K$-bi-invariant measure $\mu^K \in M(G)^{K}$ defined by
\begin{equation*}\label{Kaverage}
    \mu^{K}(A) := \int_{K}\int_{K}\mu(kAk') \,\textup{d}\lambda_{K}(k) \,\textup{d}\lambda_{K}(k').
\end{equation*}
Identifying $\ff \in C(G)$ with the absolutely continuous measure $\mu_{\ff} = \ff \,\mbox{d}\lambda_{G}$ it is easy to see that if $\ff$ is positive definite, so is the function $\ff^K$ associated with the measure $\mu^K$ by $\mu^K = \ff^K \,\mbox{d}\lambda_G$. In fact, we have the following simple lemma.

\begin{lemma}\label{Ksymmetrisation} Let $G$ be a compact group, $K$ a closed subgroup of $G$, and $U$ a $K$-bi-invariant subset of $G$. Let $\ff \in C(G)$, then $\ff^K \in C(G)^K$ and the following statements hold. 
\begin{itemize}
    \item [(i)] If $\ff|_{U} \le 0$ then $\ff^K|_U \le 0$, and if $\ff|_{U} \ge 0$ then $\ff^K|_U \ge 0$.
    \item [(ii)] If $\ff$ is positive definite, then so is $\ff^K$.
    \item [(iii)] We have the equality of the integrals
    \begin{equation*}
        \int_{G} \ff^{K}(g) \,\textup{d}\lambda_{G}(g) =   \int_{G} \ff(g) \,\textup{d}\lambda_{G}(g).
    \end{equation*}
\end{itemize}
\end{lemma}

\begin{proof}

The continuity of $\ff^K$ follows from the uniform continuity of $\ff$ on the compact group $G$.
Note that if $U$ is $K$-bi-invariant and $\ff \le 0$ on $U$, then it is straightforward that
\begin{equation*}%\label{Kaverage}
    \ff^{K}(g) := \int_{K}\int_{K} \ff(kgk')\textup{d}\lambda_{K}(k)\textup{d}\lambda_{K}(k') \le 0
\end{equation*}
for $g \in U$. Similarly for $\ff^K \ge 0$ on $U$.

Statements (ii) and (iii) were proven in \cite[Lemma 9]{ramabulana}.
\end{proof}

A \emph{spherical function} for the Gelfand pair $(G,K)$ is a continuous function $\gamma: G \to \mathbb{C}$ such that 
\begin{equation*}
    \gamma(g_1)\gamma(g_2) = \int_{K}\gamma(g_1kg_2) \,\mbox{d}\lambda_{K}(k), \quad \gamma(e)=1,
\end{equation*}
for all $g_1, g_2 \in G$. By \cite[Theorem 8.2.6]{wolf}, a spherical function \(\gamma: G \to \mathbb{C}\) for \((G,K)\) is necessarily \(K\)-bi-invariant.

A spherical function for a compact Gelfand pair is necessarily positive definite \cite[Theorem 9.10.1]{wolf}. The \emph{spherical dual} or \emph{dual space} of the Gelfand pair \((G,K)\) is defined as the collection $\Gamma$ of spherical functions for $(G,K)$. Since \(G\) is compact, \(\Gamma\) is discrete \cite[Proposition 2.4]{bergetal}. Moreover, on $\Gamma$ we put the Plancherel measure \cite[Theorem 2.1]{bergetal} or \cite[Theorem 9.4.1]{wolf}. Since $\Gamma$ is discrete, the Plancherel measure reduces to a purely atomic measure with weights given in Theorem \ref{expansion} below.

The Fourier transform of \(\ff \in L^1(G)^K\) is defined as
\begin{equation*}
    \widehat{\ff}(\gamma):=\int_{G} \ff(g) \overline{\gamma(g)} \,\mbox{d}\lambda_{G}(g), \quad \gamma \in \Gamma.
\end{equation*}
The Fourier transform of \(\mu \in M(G)^K\) is defined as
\begin{equation*}
    \widehat{\mu}(\gamma):=\int_{G} \overline{\gamma(g)} \,\mbox{d}\mu(g), \quad \gamma \in \Gamma.
\end{equation*}
Consistent with positive definite measures from Section \ref{section:introduction} where $G$ is a compact Abelian group, a $K$-bi-invariant measure $\mu \in M(G)^{K}$ is called positive definite if for all continuous ``test functions'' $u\in C(G)^{K}$ the integral of $\mu$ against the convolution square $u \star \widetilde{u}$ is non-negative: $\int_G u \star \widetilde{u} \,\mbox{d} \mu \ge 0.$ As spherical functions are their own convolution squares \cite[Theorem 8.2.6]{wolf}, this in particular means that $\widehat{\mu}(\gamma) \ge 0$ $(\gamma \in \Gamma)$, so that the Fourier transform is non-negative. 
Recall (see \cite[Proposition 8.4.6, Theorem 8.4.8]{wolf}) that for every continuous positive definite \(K\)-bi-invariant function \(\ff: G \to \mathbb{C}\) we can associate an irreducible unitary representation \(\pi_\ff: G \to GL(H_{\pi_\ff})\) on a Hilbert space \(H_{\pi_\ff}\) and a vector \(u \in H_{\pi_\ff}\) such that \(\ff(g)=(u, \pi_\ff(g)u)\) for all \(g \in G\). Moreover, the representation \(\pi_\ff\) is unique up to equivalence of representations. Since \(G\) is compact and \(\pi_\ff\) is irreducible, it is well-known that \(\pi_\ff\) is a finite-dimensional representation. We put \(\delta(\ff)= \deg(\pi_\ff) = \dim(H_{\pi_\ff})\).
We shall need the following theorem.

\begin{theorem}[{\cite[Proposition 9.10.4]{wolf}, \cite[Theorem 2.6]{bergetal}}]\label{expansion}
Let \((G,K)\) be a compact Gelfand pair. The family \(\{\sqrt{\delta(\gamma)}\gamma: \gamma \in \Gamma\}\) forms an orthonormal basis for the space of square integrable complex-valued $K$-bi-invariant functions. In particular, each complex-valued square integrable $K$-bi-invariant function has the orthogonal expansion
\[
\ff  = \sum_{\gamma \in \Gamma}\delta(\gamma)\widehat{\ff}(\gamma)\gamma,
\]
which converges in the $L^2$ norm. Moreover, a continuous $K$-bi-invariant function \(\ff: G \to  \CC\) is positive definite if and only if there exists a family \((B(\gamma))_{\gamma \in \Gamma}\) of non-negative numbers satisfying \(\sum_{\gamma \in \Gamma}B(\gamma) < \infty\) such that 
\begin{equation}\label{uniformconvergence}
\ff(g) = \sum_{\gamma \in \Gamma}B(\gamma)\gamma(g), \quad g\in G,
\end{equation}
and \(B(\gamma) = \delta(\gamma)\widehat{\ff}(\gamma)\). The series in \eqref{uniformconvergence} converges uniformly on \(G\).
\end{theorem}

If $f \in \ell^1(\Gamma)$, we define the inverse Fourier transform
$$
\widecheck{f}(g) := \sum_{\gamma \in \Gamma} \delta(\gamma)f(\gamma) \gamma(g), \quad g \in G.
$$
Observe that $\delta(\mathbbm{1}_G) = 1$. One way to see this is to consider the trivial representation \(\pi: G \to GL(\mathbb{C})\) and note that for \(z \in \mathbb{C}\) with \(|z|=1\), we have 
\begin{equation*}
  \mathbbm{1}_{G}(g) ~ = ~(z,\pi(g)z) = ~ (z,z) ~ = z\overline{z} = |z|^2 = 1  
\end{equation*}
for all \(g \in G\).  

Let $\Omega_+$, $\Omega_-$ be $K$-bi-invariant symmetric Borel subsets of $G$ with $e \in \intt{\Omega_+}$. We consider the function class
$$
\mathcal{F}^K_G(\Omega_+, \Omega_-) := \{ \varphi \in C(G)^{K} : \varphi  \gg 0, \ \varphi(e) = 1, \  \varphi|_{\Omega_+^c} \le 0, \  \varphi|_{\Omega_-^c} \ge 0 \}
$$
and the extremal value 
\begin{equation} \label{eq:Delsarte_type_problem}
\mathcal{C}^K_G(\Omega_+, \Omega_-) := \sup \left\{ \int_G \varphi \,\mbox{d}\lambda_G : \varphi \in \mathcal{F}^K_G(\Omega_+, \Omega_-) \right\}. 
\end{equation}
We term this problem the Delsarte-type problem. Problem~\eqref{eq:Delsarte_type_problem} in the setting where $G$ is a not necessarily compact LCA group with identity $0$ and $K =\{0\}$ was studied by Berdysheva, R\'ev\'esz  in~\cite{Berdysheva} and Berdysheva, Ramabulana, R\'ev\'esz in \cite{BRR_Expo}. 
Problem~\eqref{eq:Delsarte_type_problem} includes as particular cases the Tur\'an and Delsarte problems discussed in Section~\ref{section:introduction}. Namely,
if we set $\Omega_+ = \Omega_- = \Omega$ in~\eqref{eq:Delsarte_type_problem}, we arrive at the Tur\'an problem
$$
\mathcal{T}^K_G(\Omega)  : = \mathcal{C}^K_G(\Omega, \Omega) = \sup \left\{ \int_G \varphi \,\mbox{d}\lambda_G : 
 \varphi  \gg 0, \ \varphi(e) = 1, \  \varphi|_{\Omega^c} \equiv 0 \right\}. 
$$
The choice $\Omega_- = G$ (i.e. there is no restriction on the set of negativity of $\varphi$) recovers the Delsarte problem
$$
\mathcal{D}^K_G(\Omega_+)  := \mathcal{C}^K_G(\Omega_+, G) = \sup \left\{ \int_G \varphi \,\mbox{d}\lambda_G :
\varphi  \gg 0, \ \varphi(e) = 1, \  \varphi|_{\Omega_+^c} \le 0 \right\}. 
$$
The problem \eqref{eq:Delsarte_type_problem} in the case where $G$ is non-commutative was introduced by Ramabulana in \cite{ramabulana}. 

Observe that our restrictions on the sets $\Omega_\pm$ are natural. A real-valued positive definite function is always even, therefore it is natural to take the sets $\Omega_\pm$ symmetric. The second assumption $e \in \intt{\Omega_+}$ arises from the fact that the class $\mathcal{F}^K_G(\Omega_+, \Omega_-)$ is non-empty if and only if $e \in \intt{\Omega_+}$. Indeed, if $\ff(e) = 1$, then by continuity of $\ff$ there exists a neighbourhood of $e$ where $\ff$ is positive; therefore, $e$ cannot lie in $\Omega_+^c$ or on the boundary of $\Omega_+$. On the other hand, if $e \in \intt{\Omega_+}$ we can take a $K$-bi-invariant symmetric neighbourhood $V$ of \(e\) such that $V^2 \subset \Omega_+$. Then the  self-convolution $\mathbbm{1}_V \star \mathbbm{1}_V$ is a positive definite, non-negative $K$-bi-invariant function with support in $V^2  \subset \Omega_+$ and with the value at the origin $\mathbbm{1}_V \star\mathbbm{1}_V (e) = \lambda_G(V) > 0$. Therefore, $\frac{1}{\lambda_G(V)} \mathbbm{1}_V \star \mathbbm{1}_V \in \mathcal{F}^K_G(\Omega_+, \Omega_-)$. Finally, the assumption that $\Omega_{\pm}$ are $K$-bi-invariant is natural since for a $K$-bi-invariant function $\ff: G \to \mathbb{R}$, we have that $\ff(g) \ge 0$ ($\ff(g) \le 0$) if and only if $\ff(kgk') \ge 0$ ($\ff(kgk') \le 0$) for all $k,k' \in K$.

%%%%%%%%%%%%%%%%%%%%%%%%%%%%%%%%%%%%%%%%%%%%%%%%%%%%%%%%%%%%%%
%%%%%%%%%%%%%%%%%%%%%%%%%%%%%%%%%%%%%%%%%%%%%%%%%%%%%%%%%%%%%%

\section{Duality in infinite dimensional linear programming: Duffin's approach}\label{section:duality_Duffin}

The aim of this section is to put together---in the form we need---known results concerning problems of infinite dimensional linear programming, and their duality. Duffin~\cite{Duffin} seems to be the first to consider duality relations not only between the values of the primal and the dual problems, but allowing a linear programming problem with weakened restrictions on one of the sides.  His ideas were generalized to a much more general setup (convex programming and even more general problems, and replacing topological duality of spaces by duality of linear spaces induced by a bilinear form), in particular, by  Ioffe and Tikhomirov in~\cite{I-T}, and by Gol'shtein in a number of papers and the book~\cite{golstein} (the latter is unfortunately available only in Russian). The three mentioned sources are the foundation for this section.

Here we will mainly follow the presentation in Duffin's paper \cite{Duffin}, but use notation and terminology adapted to our goals.

Let $E$ be a real linear space with a locally convex topology. Let $E^*$ be its topological dual, i.e., the space of continuous linear functionals on $E$. We denote by $\langle y, x \rangle_1$ the action of $x \in E^*$ on $y \in E$. Note that the bilinear form $\langle \cdot, \cdot \rangle_1$ puts the spaces $E$, $E^*$ in duality in the sense of duality of linear spaces with respect to a bilinear form\footnote{We say that real linear spaces $C$ and $D$ are \emph{in duality} with respect to a bilinear form $\langle \cdot, \cdot \rangle : C \times D \to \RR$, if the following two separation properties hold:

\begin{itemize}

\item[1)] $\forall \, y \in C \setminus \{ 0 \} \ \exists \, x \in D : \langle y, x \rangle \ne 0$,

\item[2)] $\forall \, x \in D \setminus \{ 0 \} \ \exists \, y \in C : \langle y, x \rangle \ne 0$.

\end{itemize}

In this sense of duality, the roles of the spaces $C$ and  $D$ are symmetric.
}.

We select in $E$ an arbitrary closed convex cone which we will denote by $P$ and term it the \emph{positive cone}. We write $y \ge_{P} 0$ if $y \in P$, and $y_1 \ge_{P} y_2$ if $y_1 - y_2 \in P$. The relation $\ge_{P}$ is a preorder on the space $E$. The positive cone $P$ of the space $E$ induces in $E^*$ the \emph{dual cone}
$$
P^* := \{ x \in E^* : \langle y, x \rangle_1 \ge 0 \ \text{for all} \ y \in P \}.
$$
The set $P^*$ is a convex cone. We take $P^*$ for the positive cone of $E^*$, with the corresponding preorder $\ge_{P^*}$ on $E^*$. 

We take a second real linear space $F$ with a locally convex topology. Let $F^*$ be its topological dual. We denote by $\langle z, w \rangle_2$ the action of $w \in F^*$ on $z \in F$. Let a closed convex cone $Q$ be the positive cone in $F$, and its dual cone $Q^*$ be the positive cone in $F^*$.  

Let $T : E^* \to F$ be a linear operator, $b \in F$ and $c \in E$. The main object of this section is the \emph{linear programming problem}
\begin{align} \label{eq:LP-prim-problem}
    u = & \inf \, \langle c,x \rangle_1 \\
    \text{subject to} \quad & x \ge_{P^*} 0,  \nonumber \\
    & Tx  \ge_{Q} b. \nonumber
\end{align}

\begin{definition}

An element $x \in E^*$ is called \emph{feasible} if $x \ge_{P^*} 0$ and $Tx  \ge_{Q} b$. Problem~\eqref{eq:LP-prim-problem} is called \emph{consistent} if there exists at least one feasible $x \in E^*$. If the problem is consistent, the value
$$
u = \inf \{ \langle c,x \rangle_1 : x \ge_{P^*} 0, \  Tx  \ge_{Q} b\}
$$
is called the \emph{value} of problem~\eqref{eq:LP-prim-problem}.

\end{definition}

Note that an element \(x \in E^*\) is feasible if and only if $x \ge_{P^*} 0$ and \(Tx = b + q\) for some \(q \ge_{Q} 0\). With this in mind, the following definition makes sense.

\begin{definition} \label{def:asymptotic-value}

A sequence $(x_n) \subset E^*$ is called \emph{asymptotically feasible} if $x_n \ge_{P^*} 0$ and
$$
T x_n = b + q_n + z_n, \quad \text{where} \quad q_n \ge_{Q} 0, \ z_n \to 0.
$$
Problem~\eqref{eq:LP-prim-problem} is called \emph{asymptotically-consistent} if there exists at least one asymptotically feasible sequence $(x_n) \subset E^*$. If the problem is asymptotically-consistent, the value
\begin{equation} \label{eq:asymptotic-value}
u_a := \inf \{ \liminf_{n} \, \langle c, x_n \rangle_1 : (x_n) \text{ is an asymptotically feasible sequence}  \}
\end{equation}
is called the \emph{asymptotic-value} of problem~\eqref{eq:LP-prim-problem}. 

\end{definition}

A simple diagonal argument shows that in~\eqref{eq:asymptotic-value} it is sufficient to take the infimum over all asymptotically feasible sequences $(x_n)$ such that the limit $\lim_n \, \langle c, x_n \rangle_1$ exists.

Duffin~\cite{Duffin} calls asymptotically-consistent problems sub-consistent, and asymptotic-values sub-values. Ioffe and Tikhomirov~\cite{I-T} call asymptotically-consistent problems weakly consistent, and asymptotic-values weak values. Gol'shtein~\cite{golstein} calls asymptotically feasible sequences generalized plans, and the problem $u_a$ the generalized problem. Ioffe and Tikhomirov~\cite{I-T} and Gol'shtein~\cite{golstein} work in somewhat different setups that differ from our presentation---and that of Duffin---by choices of topologies and duality relations between the spaces in each pair. Ioffe and Tikhomirov~\cite{I-T} work in Definition~\ref{def:asymptotic-value} with nets.  Both \cite{Duffin} and \cite{golstein} work with sequences, but both authors mention that their results can be extended verbatim to nets when working with general locally convex spaces (and not, say, with metric spaces). We restrict ourselves to sequences as this will be sufficient for our application (in our application the space $F$ will be the normed space of continuous functions on a compact group).

It is straightforward that if a problem is consistent, then it is also asymptotically-consistent and
$$
u_a \le u.
$$

\begin{definition}

Problem ~\eqref{eq:LP-prim-problem} is called \emph{well-posed} if
$$
u_a = u.
$$

\end{definition}

Well-posedness is not defined in Duffin~\cite{Duffin}. The above definition follows Ioffe and Tikhomirov~\cite{I-T}. Gol'shtein~\cite{golstein} calls such problems correctly posed. He defines this property in different terms, but proves that his definition can be reduced to considering asymptotically feasible sequences. Note that the setups in~\cite{I-T} and~\cite{golstein} are slightly different from ours.

We assume that the linear operator $T : E^*  \to F$ has an adjoint operator in the following sense: this is a linear operator $T^* : F^* \to E$ with the defining property
$$
\langle T^*w, x \rangle_1 = \langle Tx, w \rangle_2  \quad \text{for all} \quad x \in E^*, \ w \in F^*.
$$
Note that $T^*$ always exists as a linear operator from $F$ to $E^{**}$. Here we assume more, namely, that $T^*w \in E$. 
In our application below in Section~\ref{section:A-B_scheme}, such adjoint operator will exist. Note that we do not assume continuity of either of the operators $T$, $T^*$ (with respect to the preliminarily given topologies).

The dual problem of~\eqref{eq:LP-prim-problem} is
\begin{align} \label{eq:LP-dual-problem}
    v = & \sup \, \langle b, w \rangle_2 \\
    \text{subject to} \quad &  w \ge_{Q^*} 0,  \nonumber \\
    & T^*w  \le_{P} c. \nonumber
\end{align}

The dual problem of~\eqref{eq:LP-dual-problem} is again \eqref{eq:LP-prim-problem}. In Duffin's setup, the roles of \eqref{eq:LP-prim-problem} and \eqref{eq:LP-dual-problem} are symmetric.

The main statement of this section is the following duality theorem.

\begin{theorem} \label{theorem:Duffin-duality}

The problem of linear programming~\eqref{eq:LP-prim-problem} is asymptotically-consistent and has a finite asymptotic-value if and only if the dual problem~\eqref{eq:LP-dual-problem} is consistent and has a finite value. Moreover, in this case
$$
u_a = v.
$$

\end{theorem}

This statement is Theorem~1 in~\cite{Duffin}. Statements which only slightly differ by assumptions on the spaces are particular cases of more general Theorem~2.1 and Corollary in~\cite{I-T}, and of Theorem~3.1 in Chapter~2 of~\cite{golstein}. Moreover, it is proven in~\cite{I-T} that their (more general than Theorem~\ref{theorem:Duffin-duality}) result is equivalent to the Fenchel-Moreau Theorem (see Theorem 1.1 in~\cite{I-T}). 

\begin{corollary} \label{corollary:Duffin-strong-duality}

If problem~\eqref{eq:LP-prim-problem} is well-posed, then the strong duality holds:
$$
u = v.
$$

\end{corollary}
%%%%%%%%%%%%%%%%%%%%%%%%%%%%%%%%%%%%%%%%%%%%%%%%%%%%%%%%%%%%%%
%%%%%%%%%%%%%%%%%%%%%%%%%%%%%%%%%%%%%%%%%%%%%%%%%%%%%%%%%%%%%%

\section{The dual cone of the cone $Q$} \label{section:Q-polar}

Recall that $\Omega_+$, $\Omega_-$ are two $K$-bi-invariant symmetric Borel subsets of $G$ with $e \in \intt{\Omega_+}$. We denote $A:=\Omega_+^c$ and $B:=\Omega_-^c$.

Our study will naturally lead to the consideration of the convex closed cone $Q = M \cap L$ in the space $C(G)^{K}$, where
\begin{equation} \label{eq:K_and_L}
M  := \{ \ff \in C(G)^K~:~ \ff|_A \le 0\} \quad \text{and} \quad L := \{ \ff\in C(G)^K~:~ \ff|_B\ge 0\}.
\end{equation}
By continuity of $\ff$ it follows that if $\ff \le 0$ on $A$, then $\ff \le 0$ also on the closure $\overline{A}$. A similar observation is of course true for $B$. Therefore,
$$
M = \{ \ff \in C(G)^K~:~ \ff|_{\overline{A}} \le 0\}, \quad L = \{ \ff\in C(G)^K~:~ \ff|_{\overline{B}}\ge 0\}.
$$

Since $G$ is compact, the dual of $C(G)$ is the space  $M(G)$ of finite, regular signed Borel measures on $G$. Since the action of any $\mu \in M(G)$ on $\ff \in C(G)^K$ depends on its average over double cosets $KgK$, we may replace $\mu$ by $\mu^K$, and we have
\begin{equation*}
    \int_{G} \ff(g) \,\mbox{d} \mu(g) = \int_{G} \ff(g) \,\mbox{d}\mu^{K}(g) \text{ for all } \ff \in C(G)^K.
\end{equation*}
In other words, the dual of $C(G)^K$ is the space $M(G)^K$.\footnote{Note that if we do not restrict \(M(G)\) to \(M(G)^K\), then the bilinear form for the pair \((C(G)^K, M(G))\) does not put the pair in duality since the measures \(\mu\) and \(\mu^K\) give the same value on any \(\ff \in C(G)^K\).}

The goal of the section is to determine the dual cone $Q^*$ in the space $M(G)^K$.  We write
$$
\langle \varphi, \mu \rangle_2 = \int_{G} \varphi \,\mbox{d}\mu
$$
for the action on $\varphi \in C(G)^K$ of $\mu \in M(G)^K$. 
 
Let us first prove an easy extension of Urysohn's lemma to $K$-bi-invariant functions on a locally compact groups with a compact subgroup $K$.

\begin{lemma}[$K$-bi-invariant Urysohn's Lemma]\label{inv_urysohn} Let $G$ be a locally compact group and $K$ a compact subgroup of $G$. Then for any disjoint closed $K$-bi-invariant subsets $A$ and $B$ of $G$ there exists a continuous $K$-bi-invariant function $s:G \to [0,1]$ such that $s{|}_A\equiv 1$ and $s|_{B} \equiv 0$.
\end{lemma}

\begin{proof}
Since every locally compact group is normal \cite[Chapter IV, Theorem 1]{husain}, by Urysohn's Lemma (e.g.~\cite[Theorem~7.2.5]{Richmond}), there is a continuous function $\ff: ~G \to~ [0,1]$ such that $\ff|_{A} \equiv 1$ and $\ff|_{B} \equiv 0$. Furthermore, since $A$ is $K$-bi-invariant, we have that for any $g \in A$, $KgK \subset A$. Hence, for any $g \in A$ we have
\begin{equation*}
    \ff^{K}(g)= \int_{K}\int_{K} \ff(kgk') \,\mbox{d}\lambda_{K}(k) \,\mbox{d}\lambda_{K}(k') = \int_{K}\int_{K} 1 \,\mbox{d}\lambda_{K}(k) \,\mbox{d}\lambda_{K}(k') = 1.
\end{equation*}
Hence \(\ff^K|_{A} \equiv 1\). Similarly, \(\ff^K|_{B} \equiv 0\).

Furthermore, since $|\ff(g)| \le 1$ for all $g \in G$, and $\lambda_{K}(K) = 1$, we have 
\begin{equation*}
    |\ff^{K}(g)| \le \int_{K}\int_{K} |\ff(kgk')| \,\mbox{d}\lambda_{K}(k) \,\mbox{d}\lambda_{K}(k') = 1,
\end{equation*}
so that $\ff^{K}(G) \subset  [0,1]$.
Put $s = \ff^{K}$.
\end{proof}

\begin{lemma} \label{lemma:K-polar}

Let $G$ be a compact group, $K$ a closed subgroup of $G$, and $A \subset G$ a $K$-bi-invariant Borel set. The dual cone of $M = \{ \ff \in C(G)^K~:~ \ff|_A \le 0\}$ in the space $M(G)^K$ is the cone
\begin{equation}\label{eq:K-polar}
M^* := \{ \mu \in M(G)^K :  \langle \ff, \mu \rangle_2 \ge 0  \ \forall \ff \in M\} 
= \{ \mu \in M(G)^K :  \mu|_{\overline{A}} \le 0 \text{ and } \mu|_{\overline{A}^c}\equiv 0 \}. 
\end{equation}

\end{lemma}

\begin{proof}
For one direction, note that if $\mu|_{\overline{A}^c}\equiv 0$ and $\mu|_{\overline{A}} \le 0$, then for any $\ff \in M$ we find $\langle \varphi, \mu \rangle_2 = \int_G \ff \,\mbox{d}\mu  = \int_{\overline{A}} \ff \,\mbox{d}\mu \ge 0$, since $\ff|_{\overline{A}}\le 0$. So, $M^*$ contains the right-hand side of~\eqref{eq:K-polar}. To prove the converse containment, we have to prove two claims for an arbitrary $\mu \in M^{*}$.

First, consider the assertion that $\mu|_{\overline{A}^c} \equiv 0$. Here we refer to inner regularity of the measure $\mu$: it suffices to show that for any $K$-bi-invariant compact set $F \Subset \overline{A}^c$, $\mu(F)=0$. As $\overline{A}^c$ is $K$-bi-invariant and open, there are open $K$-bi-invariant sets  $U$ inscribed between $F$ and $\overline{A}^c$, i.e., $F \Subset U \subset \overline{A}^c$. 
Referring to outer regularity of $\mu$, it is possible to choose a $K$-bi-invariant $U$ with $|\mu|(U\setminus F)<\ve$, where $\ve > 0$ is an arbitrary positive number. Note that $U^c$ is $K$-bi-invariant and disjoint from the $K$-bi-invariant~$F$. Therefore, by Lemma \ref{inv_urysohn}, there exists a $K$-bi-invariant $s \in C(G)$, $s : G \to [0,1]$ such that $s|_F\equiv 1$ and $s|_{U^c}\equiv 0$, in particular, $s|_{\overline{A}} \equiv 0$. As a result, $s \in M$ and $-s \in M$, and therefore $\langle s, \mu \rangle_2 \ge 0$ and $\langle -s, \mu \rangle_2 \ge 0$ which implies $\langle s, \mu \rangle_2 = 0$. Consider now $0 = \langle s, \mu \rangle_2 = \int_U s \,\mbox{d}\mu = \int_F s \,\mbox{d}\mu + \int_{U\setminus F} s \,\mbox{d}\mu$. As $s|_F \equiv 1$, the first term is just $\mu(F)$. The second term can be estimated as $\|s\|_\infty |\mu|(U\setminus F)$.  We therefore obtain $|\mu(F)| \le |\mu|(U\setminus F) <\ve$, and as $\ve$ was arbitrarily small, we are led to $\mu(F) = 0$. So, by inner regularity, also $\mu|_{\overline{A}^c} \equiv 0$.

Second, we claim that $\mu|_{\overline{A}} \le 0$. Again, we take an arbitrary $\ve > 0$ and any pair of $K$-bi-invariant sets, $F \Subset A$ compact and $U\supset F$ open, with $|\mu|(U\setminus F) < \ve$. Here we cannot guarantee $U \subset A$, but as we already know $\mu|_{\overline{A}^c}\equiv 0$, this does not bother us. Again, we take a $K$-bi-invariant function $s \in C(G)$, $s : G \to [0,1]$ such that $s|_F\equiv 1$ and $s|_{U^c}\equiv 0$. From $s \ge  0$ we see that $-s \in M$. Therefore, $0 \le \langle -s, \mu \rangle_2 = \int_U (-s) \,\mbox{d}\mu = \int_F (-s) \,\mbox{d}\mu + \int_{U\setminus F} (-s) \,\mbox{d}\mu$, and the second term is at most $\|s\|_\infty |\mu|(U\setminus F) < \ve$ again. As the first expression is just $-\mu(F)$, we find $\mu(F) < \ve$. Since $\ve >0$ is arbitrary, we conclude that $\mu(F)\le 0$, and by inner regularity we obtain $\mu|_A \le 0$.
\end{proof}

Similarly, 
$$
L^* = \{\mu \in M(G)^K :  \mu|_{\overline{B}} \ge 0 \text{ and } \mu|_{\overline{B}^c}\equiv 0 \}.
$$
Recall that $A^c = \Om_{+}$ and $B^c = \Om_{-}$, so that the measures in $M^*$ live only in $\overline{\Omega_+^c}$ and the ones in $L^*$ live only in $\overline{\Omega_-^c}$, respectively. Note that if $K=\{0\}$, then in the compact (Abelian) group case we get back $M^*=-M^*(\Omega_{+})$ and $L^*=M^*(\Omega_{-})$ appearing in the formulation of Theorem \ref{theorem:main_group}.

Now we are in a position to describe $Q^* := (M \cap L)^*$, for which we shall need the following lemma.

\begin{lemma} \label{lemma:Jordan}
Let \(G\) be a compact group and \(K\) a closed subgroup. Any \(\mu \in M(G)^K \) has Jordan decomposition \(\mu = \mu_ + - \mu_ -\) with \(\mu _ +, \mu _ - \in M(G)^K\).
\end{lemma}

\begin{proof}
    Following \cite[Theorem 4.1.5, Corollary 4.1.6]{{cohn-measure-theory}}, let \((P,N)\) be a Hahn decomposition for \(\mu\) so that \(\mu_{+}(E) = \mu(E \cap P)\) and \(\mu_{-}(E) = -\mu(E \cap N)\) for all measurable subsets \(E \subset G\). Let \(k\) and \(k'\) be arbitrary elements of \(K\). From the fact that \(\mu\) is \(K\)-bi-invariant and that \(G \ni g \mapsto kgk' \in G\) is a homeomorphism, it is clear that \((kPk',kNk')\) is also a Hahn decomposition for \(\mu\). Therefore, \(\mu(E \cap P) = \mu(E\cap kPk')\) and \(\mu(E \cap N) = \mu(E\cap kNk')\)  for all measurable subsets \(E \subset G\). Now
    \begin{equation*}
        \mu_{+}(kEk') = \mu((kEk')\cap P) = \mu((kEk')\cap (kPk')) = \mu(k(E\cap P)k') = \mu(E \cap P) = \mu_{+}(E)
    \end{equation*}
for all measurable subsets \(E \subset G\). Hence \(\mu_{+}\) is \(K\)-bi-invariant. Similarly, \(\mu_{-}\) is \(K\)-bi-invariant.
\end{proof}

\begin{theorem} \label{theorem:Q-polar}

Let $G$ be a compact group, $K$ a closed subgroup of $G$, and $A,B \subset G$ $K$-bi-invariant Borel sets. The dual cone $Q^*$ of $Q = M \cap L$, where $M$ and $L$ are defined in~\eqref{eq:K_and_L} is the cone $M^* + L^*$.

\end{theorem}

\begin{proof}
It is obvious that $M^* + L^* \subset Q^*$. We will show the converse containment.

Let $\mu \in Q^*$, and let us take its Jordan decomposition $\mu = \mu_{+} - \mu_{-}$ with $K$-bi-invariant $\mu_{+}$ and $\mu_{-}$ that exists by Lemma~\ref{lemma:Jordan}. We will show that $-\mu_{-} \in M^{*}$, and then similarly $\mu_{+} \in L^*$, furnishing the required representation of the general $\mu \in Q^*$. 

So we are going to prove that $\supp \mu_{-} \subset \overline{A}$, in other words, $\mu_{-}|_{\overline{A}^c} \equiv 0$, or, still equivalently, $\mu|_{\overline{A}^c} \ge 0$. A similar argument will imply that $\supp \mu_{+} \subset \overline{B}$, or, equivalently, $\mu|_{\overline{B}^c} \le 0$. Let us see that these properties entail the desired property that $-\mu_{-} \in M^{*}$, and $\mu_{+} \in L^{*}$. Indeed, if $\ff \in M$, then we have $\langle \ff, -\mu_{-} \rangle_2 = - \int_{\overline{A}} \ff \,\mbox{d}\mu_{-} \ge 0$, for $\ff|_{\overline{A}} \le 0$ and $\mu_{-} \ge 0$. This means that $-\mu_{-} \in M^{*}$, and similarly $\mu_{+} \in L^{*}$. Thus, we would obtain $Q^* \subset M^* + L^*$, and finally $Q^* = M^* + L^*$.

So let now prove that $\mu|_{\overline{A}^c} \ge 0$. Let $F \Subset \overline{A}^c$ be an arbitrary $K$-bi-invariant subset. As $\overline{A}^c$ is open and $\mu$ is regular, to any $\ve>0$ there exists an open $K$-bi-invariant subset  $U$  such that $F \Subset U \subset \overline{A}^c$ and $|\mu|(U\setminus F) < \ve$. Using again Lemma \ref{inv_urysohn}, we can get a function $s \in C(G)^{K}$, $s : G \to [0,1]$ with $s|_F \equiv 1$, $s|_{U^c} \equiv 0$. This function is a non-negative continuous $K$-bi-invariant function, so without doubt $s \in L$. But also $s|_{\overline{A}} \equiv 0$ as $U^c \supset \overline{A}$, so $s|_{\overline{A}}\le 0$ and thus $s \in M$, too. In all, $s \in Q = M\cap L$.

Given that $\mu \in Q^*$, we must have $0 \le \langle s, \mu \rangle_2 = \int_F s \,\mbox{d}\mu + \int_{U\setminus F} s \,\mbox{d}\mu = \mu(F) + \int_{U\setminus F} s \,\mbox{d}\mu$. Again, the second term can be estimated as $| \int_{U\setminus F} s \,\mbox{d} \mu| \le  \|s\|_\infty |\mu|(U\setminus F) < \ve$, so that we are led to $- \ve < \mu(F)$, and as this holds for all $\ve > 0$, we obtain $\mu(F)\ge 0$. That is, by inner regularity, all measurable subsets of the open subset $\overline{A}^c$ will have non-negative measure, and therefore $\mu|_{\overline{A}^c} \ge 0$, as wanted.
\end{proof}

%%%%%%%%%%%%%%%%%%%%%%%%%%%%%%%%%%%%%%%%%%%%%%%%%%%%%%%%%%
%%%%%%%%%%%%%%%%%%%%%%%%%%%%%%%%%%%%%%%%%%%%%%%%%%%%%%%%%%

\section{Arestov-Babenko Scheme on Compact Gelfand Pairs} \label{section:A-B_scheme}

As in Section~\ref{section:duality_Duffin}, we will work with two pairs of spaces in duality. For the first pair, we take $E = \ell^\infty(\Gamma)$.  It is known that  $\ell^\infty(\Gamma)$ is the topological dual of  $\ell^1(\Gamma)$ with the norm topology, if $\Gamma$ is discrete (e.g. \cite[Theorem  4.2.1]{edwards}). Keeping this fact in mind, we take $E = \ell^\infty(\Gamma)$ with the weak-$\ast$ topology $\sigma(\ell^\infty(\Gamma), \ell^1(\Gamma))$. This makes $E$ a space with a locally convex topology.  Moreover, the topological dual of $\ell^\infty(\Gamma)$ with the weak-$\ast$ topology $\sigma(\ell^\infty(\Gamma),   \ell^1(\Gamma))$ is $E^* = \ell^1(\Gamma)$ (e.g. \cite[Proposition 3.14]{Brezis}). In other words, with this choice, the spaces $E$ and $E^*$ satisfy the assumptions of Section~\ref{section:duality_Duffin}. The theory in Section~\ref{section:duality_Duffin} does not require a topology on $E^*$.

Note that for a function $f : \Gamma \to \mathbb{R}$, the condition $f \in \ell^1(\Gamma)$ implies that $\supp{f}$ is at most countable, and $\sum_{\gamma \in \supp{f}} |f(\gamma)| < \infty$.
The corresponding bilinear form is
$$
\langle g,f \rangle_1 = \sum_{\gamma \in \supp{f}} f(\gamma) g(\gamma)
$$
for  $f \in \ell^1(\Gamma)$ and $g \in \ell^\infty(\Gamma)$.  

For the second pair we take $F = C(G)^{K}$ with the norm topology and $F^*~= (C(G)^{K})^* = M(G)^{K}$, the space of regular signed $K$-bi-invariant Borel measures on $G$, with the bilinear form
$$
\langle \varphi, \mu \rangle_2 = \int_{G} \varphi \, \mbox{d}\mu
$$
for  $\varphi \in C(G)^{K}$ and $\mu \in M(G)^K$. Also in this case, the assumptions of Section~\ref{section:duality_Duffin} are satisfied.

Let $\Omega_+$, $\Omega_-$ be symmetric $K$-bi-invariant Borel subsets of $G$ with $e \in \intt{\Omega_+}$. We denote their complements by $A = \Omega_+^c$, $B = \Omega_-^c$.

As the positive cone in the space $E = \ell^\infty(\Gamma)$ we take the closed convex cone $P = \ell^\infty_+(\Gamma)$ of non-negative functions in $\ell^\infty(\Gamma)$. Its dual cone in the space $E^* = \ell^1(\Gamma)$ is the convex cone  $P^* = \ell^1_+(\Gamma)$ of non-negative elements of $\ell^1(\Gamma)$. In the space $F = C(G)^K$ we take the positive cone to be the closed convex cone $Q  = M \cap L$, where $M$ and $L$ were defined in~\eqref{eq:K_and_L}. Its dual cone $Q^*$ in $M(G)^K$ is described in Theorem~\ref{theorem:Q-polar}.

We consider the function class 
$$
\widetilde{\mathcal{F}}^K_G(\Omega_+, \Omega_-) := \{ \varphi \in C(G)^K : \varphi  \gg 0, \ \varphi \not\equiv 0, \ \varphi|_{\Omega_+^c} \le 0, \  \varphi|_{\Omega_-^c} \ge 0 \}.
$$
Note that the class $\mathcal{F}^K_G(\Omega_+, \Omega_-)$ introduced in Section~\ref{section:notation} is a subclass of functions in $\widetilde{\mathcal{F}}^K_G(\Omega_+, \Omega_-)$ with $\ff(e) = 1$.

If $f \in \ell^1_+(\Gamma)$, then it follows from $|\gamma(g)| \le 1$ for all \(g \in G\) and the compactness of~$G$ that the series $\sum_{\gamma \in \supp{f}} f(\gamma)  \gamma$ converges absolutely and uniformly on $G$, and therefore defines a continuous $K$-bi-invariant function $\ff$ on $G$, with $\varphi(e) = \sum_{\gamma \in \supp{f}} f(\gamma)$. By Theorem~\ref{expansion} we have that $\ff$ is positive definite and $f(\gamma) = \delta(\gamma) \widehat{\ff}(\gamma)$, $\gamma \in \Gamma$. We conclude that functions in the class $\widetilde{\mathcal{F}}^K_G(\Omega_+, \Omega_-)$ are exactly functions of the form
$$
\varphi(g) = \sum_{\gamma \in \supp{f}} f(\gamma) \gamma(g), \quad g \in G, \quad \text{with} \quad f \in \ell^1_+(\Gamma) \setminus \{\mathbf{0}\}
$$
such that $\varphi|_A \le 0$, $\varphi|_B \ge 0$.

 We fix a positive definite measure $\sigma \in M(G)^K$. We will put on $\sigma$ an additional condition given below.

\begin{definition} \label{def:wiener}

We say that a measure $\sigma \in M(G)^K$ satisfies \emph{Wiener's condition} if $\widehat{\sigma}(\gamma) \ne 0$ for all $\gamma \in \Gamma$.

\end{definition}

We denote $s := \widehat{\sigma} \in \ell^\infty(\Gamma)$. Let $\sigma \in M(G)^K$ be a positive definite measure that satisfies Wiener's condition. Let $\ff \not\equiv 0$ be a continuous positive definite $K$-bi-invariant function. Then $f = \delta \widehat{\ff} \in \ell^1_+(\Gamma) \setminus \{ \mathbf{0} \}$. We see that
\begin{equation} \label{eq:sigma-form-positive}
\langle \varphi, \sigma \rangle_2 = \sum_{\gamma \in \supp{f}} f(\gamma) s(\gamma) > 0.
\end{equation}
Recall that $\delta(\mathbbm{1}_G) = 1$.

Instead of the extremal constant~\eqref{eq:Delsarte_type_problem} introduced in Section~\ref{section:notation}, we will consider a more general problem
\begin{equation}\label{eq:C_with_F_tilde}
\mathcal{A}^{K,\sigma}_G(\Omega_+, \Omega_{-}) :=  \sup \left\{ \frac{ \int_G \ff \,\mbox{d} \lambda_G}{ \langle \ff, \sigma \rangle_2 }  : \varphi \in \widetilde{\mathcal{F}}^K_G(\Omega_+, \Omega_-) \right\}. 
\end{equation}
The linear version of problem~\eqref{eq:C_with_F_tilde} is
$$
\mathcal{A}^{K,\sigma}_G(\Omega_+, \Omega_{-}) = \sup \left\{\int_G \varphi \, \mbox{d}\lambda_G : \varphi \in \mathcal{F}^{K,\sigma}_G(\Omega_+, \Omega_-)  \right\},
$$
where
$$
\mathcal{F}^{K,\sigma}_G(\Omega_+, \Omega_-) := \{ \varphi \in C(G)^K : \varphi  \gg 0, \ \langle \ff, \sigma \rangle_2 = 1, \  \varphi|_{\Omega_+^c} \le 0, \  \varphi|_{\Omega_-^c} \ge 0 \}.
$$
Extremal problem~\eqref{eq:Delsarte_type_problem} is a particular case of the above problem. Recall that the normalisation in~\eqref{eq:Delsarte_type_problem} was $\ff(e) = 1$. It can be realized by the condition $\langle \ff, \delta_e^K \rangle_2 = 1$, where $\delta_e^K \in M(G)^K$ is the $K$-periodisation of the Dirac measure $\delta_e$:
$$
\delta_e^K(A) := \int_K \int_K \delta_e(kAk') \, \mbox{d} \lambda_K(k) \, \mbox{d} \lambda_K(k')
$$
for a Borel set $A$. If $\ff$ is $K$-bi-invariant, then 
$$
\langle \ff, \delta_e^K \rangle_2 = \int_K \int_K \ff(kek') \, \mbox{d} \lambda_K(k) \, \mbox{d} \lambda_K(k') = \ff(e).
$$
The Fourier transform of $\delta_e^K$ is $\widehat{\delta_e^K} = \mathbbm{1}_\Gamma$.

We have discussed in Section~\ref{section:notation} that the condition $e \in \intt{\Omega_+}$ is equivalent to the fact that the class $\widetilde{\mathcal{F}}^K_G(\Omega_+, \Omega_-)$ is non-empty. Moreover, the construction described in Section~\ref{section:notation} gives a function in the class $\widetilde{\mathcal{F}}^K_G(\Omega_+, \Omega_-)$ which is non-negative and not identically zero, and therefore has a strictly positive integral. This implies that $\mathcal{A}^{K,\sigma}_G(\Omega_+, \Omega_{-}) > 0$. It follows, in particular, that in the supremum in~\eqref{eq:C_with_F_tilde} we may only consider functions $\ff$ with $f(\mathbbm{1}_{G}) = \int_G \ff \,\mbox{d} \lambda_G > 0$.

If $\varphi \in C(G)^K$, $\ff \gg 0$, $\ff \not\equiv 0$, then by~\eqref{eq:sigma-form-positive}
$\langle \varphi, \sigma \rangle_2 \ge f(\mathbbm{1}_G) s(\mathbbm{1}_G) = \int_G \ff \,\mbox{d}\lambda_G \cdot \widehat{\sigma}(\mathbbm{1}_G) > 0$,
and
$$
0 < \frac{ \int_G \ff \,\mbox{d} \lambda_G}{ \langle \ff, \sigma \rangle_2 } \le \frac{1}{\widehat{\sigma}(\mathbbm{1}_G)}.
$$
Therefore,
$$
0 < \mathcal{A}^{K,\sigma}_G(\Omega_+, \Omega_{-}) \le \frac{1}{\widehat{\sigma}(\mathbbm{1}_G)}.
$$

The extremal problem~\eqref{eq:C_with_F_tilde} can be reformulated in terms of the Fourier transforms $f = \delta \widehat{\varphi}$ and $s = \widehat{\sigma}$ as follows:
\begin{align*}
\mathcal{A}^{K,\sigma}_{G}(\Omega_+, \Omega_{-}) = \sup & \left\{ \frac{ f(\mathbbm{1}_{G}) }{ f(\mathbbm{1}_{G}) s(\mathbbm{1}_{G}) + \sum_{\gamma \in \supp{f} \setminus \{\mathbbm{1}_G\}} f(\gamma) s(\gamma)}  :  \  \widecheck{f} \in \widetilde{\mathcal{F}}^K_G(\Omega_+, \Omega_-)  \right\} \\[1mm]
= \sup & \left\{ \frac{ 1 }{s(\mathbbm{1}_G) + \sum_{\gamma \in \supp{f} \setminus \{\mathbbm{1}_G\}} f(\gamma) s(\gamma)}  : \widecheck{f}\ \in \widetilde{\mathcal{F}}^K_G(\Omega_+, \Omega_-), \ f(\mathbbm{1}_{G})= 1  \right\}.
\end{align*}
Instead of this problem, we shall consider the equivalent problem to find
\begin{align*}
u_{\Gamma}^{\sigma}(\Omega_+, \Omega_-)  := \inf & \left\{ \sum_{\gamma \in  \supp{f} \setminus \{\mathbbm{1}_G\}} f(\gamma) s(\gamma) : \right.   f \in \ell^1_+(\Gamma), \
\varphi|_{\Omega_+^c} \le 0, \ \varphi|_{\Omega_-^c} \ge 0, \\
& \text{where }\left. \varphi = \mathbbm{1}_G + \sum_{\gamma \in \supp{f} \setminus \{\mathbbm{1}_G\}} f(\gamma) \gamma
\right\}.
\end{align*}
The two problems are connected by the equation
\begin{equation} \label{eq:C-u}
\mathcal{A}^{K,\sigma}_G(\Omega_+, \Omega_{-}) = \frac{1}{\widehat{\sigma}(\mathbbm{1}_{G}) + u_\Gamma^\sigma(\Omega_+, \Omega_-)}.
\end{equation}

We consider the operator $T : \ell^1(\Gamma) \to C(G)^K$,
$$
T f := \sum_{\gamma \in \supp{f} \setminus \{\mathbbm{1}_G\}}  f(\gamma) \gamma, \quad f \in \ell^1(\Gamma).
$$
It is a linear operator from $\ell^1(\Gamma)$ to $C(G)^K$. The problem $u_\Gamma^{\sigma}(\Omega_+, \Omega_{-})$ can be rewritten as 
$$
u_\Gamma^{\sigma}(\Omega_+, \Omega_-) = \inf \left\{ \sum_{\gamma \in  \supp{f} \setminus \{\mathbbm{1}_G\}} f(\gamma) s(\gamma) : f \in \ell^1_+(\Gamma),  \ Tf + \mathbbm{1}_G \in Q  \right\}.
$$

The adjoint operator of $T$ (in the sense explained in Section~\ref{section:duality_Duffin}) is the operator $T^* : M(G)^K \to \ell^\infty(\Gamma)$,
$$
(T^* \mu)(\gamma) :=
\begin{cases}
\int_{G} \gamma \, \mbox{d}\mu , & \gamma \ne \mathbbm{1}_G, \\
0, & \gamma = \mathbbm{1}_G.
\end{cases}
$$
Indeed, for $f \in \ell^1(\Gamma)$ and $\mu \in M(G)^{K}$ we have
\begin{align*}
\langle Tf, \mu \rangle_2
& = \int_{G} \sum_{\gamma \in  \supp{f} \setminus \{\mathbbm{1}_G\}} f(\gamma) \gamma(g) \, \mbox{d}\mu(g) \\
& = \sum_{\gamma \in  \supp{f} \setminus \{\mathbbm{1}_G\}} f(\gamma) \int_{G} \gamma(g) \, \mbox{d}\mu(g)
= \sum_{\gamma \in  \supp{f} \setminus \{\mathbbm{1}_G\}} f(\gamma)  (T^* \mu)(\gamma)
= \langle T^* \mu, f \rangle_1.
\end{align*}
With the notation $b \in C(G)^K$ being the constant function $b := -\mathbbm{1}_G$ and $c \in \ell^\infty(\Gamma)$ being
$$
c(\gamma) :=
\begin{cases}
s(\gamma), & \gamma \ne \mathbbm{1}_G, \\
0, & \gamma = \mathbbm{1}_G, 
\end{cases}
$$
we arrive at the primal problem 
$$
u_\Gamma^{\sigma}(\Omega_+, \Omega_-) = \inf \{ \langle c, f \rangle_1 : f \ge_{\ell^1_+(\Gamma)} 0, \ Tf \ge_{Q} b \}
$$
which is exactly of the form \eqref{eq:LP-prim-problem}.
The dual problem \eqref{eq:LP-dual-problem} is
$$
v_\Gamma^{\sigma}(\Omega_+, \Omega_-) = \sup \{ \langle b, \mu \rangle_2 : \mu \ge_{Q^*} 0, \ T^* \mu \le_{\ell^\infty_+(\Gamma)} c\}.
$$
Here
$$
\langle b, \mu \rangle_2 = \langle -\mathbbm{1}_{G}, \mu \rangle_2 = -\int_{G} \mbox{d}\mu = - \mu(G).
$$
The condition $T^* \mu \le_{\ell^\infty_+(\Gamma)} c$ means that for $\gamma \ne \mathbbm{1}_{G}$ we have $\int_G \gamma \,\mbox{d}\mu \le s(\gamma)$. Thus, the dual problem can be written as
$$
v_\Gamma^{\sigma}(\Omega_+, \Omega_-) = \sup \left\{ -\mu(G) : \mu \in Q^*, \ \int_G \gamma \,\mbox{d}\mu \le s(\gamma) \text{ for all } \gamma \ne \mathbbm{1}_G\right\}.
$$

Our next aim is to establish the strong duality relation $u_\Gamma^{\sigma}(\Omega_+, \Omega_-) = v_\Gamma^{\sigma}(\Omega_+, \Omega_-)$ using Theorem~\ref{theorem:Duffin-duality} and Corollary~\ref{corollary:Duffin-strong-duality}. We will even show that the primal problem
$u_\Gamma^{\sigma}(\Omega_+, \Omega_-)$ is consistent, has a finite value and is well-posed. 

We start with the first two claims. We have discussed above that the class \linebreak $\widetilde{\mathcal{F}}^K_G(\Omega_+, \Omega_-)$ is non-empty; moreover, it is sufficient to consider only functions with $f(\mathbbm{1}_{G}) = \int_G ~ \ff \,\mbox{d} \lambda_G = ~1$. Fourier transforms $f = \delta \widehat{\varphi}$ of such functions are feasible in the problem $u_\Gamma^{\sigma}(\Omega_+, \Omega_-)$. Thus, the problem $u_\Gamma^{\sigma}(\Omega_+, \Omega_-)$ is consistent. Since it is a minimization problem of a non-negative quantity, we conclude that it has a finite value.

To show that the problem $u_\Gamma^{\sigma}(\Omega_+, \Omega_-)$ is well-posed, we consider the extended problems 
$$
u_\Gamma^{\sigma}(\Omega_+, \Omega_-;\varepsilon) := \inf \left\{ \sum_{\gamma \in \supp{f} \setminus \{ \mathbbm{1}_G \}} f(\gamma) s(\gamma) : f \in \ell^1_+(\Gamma), \
Tf + \mathbbm{1}_G \in Q + G(\varepsilon)   \right\}
$$
with $0 < \varepsilon < 1$ and $G(\varepsilon) := \{ g \in C(G)^K : \|g\|_\infty \le \varepsilon \}$. 
We use here the fact that the sets $G(\varepsilon)$ build a basis of neighbourhoods of zero in the topology of $C(G)^K$. If $(f_n)$ is an asymptotically feasible sequence, then $(f_n) \in \ell^1_+(\Gamma)$ and
$$
Tf_n + \mathbbm{1}_G = q_n + z_n, \quad \text{where} \quad q_n \in Q, \ \|z_n\|_{\infty} \to 0,
$$
i.e., $Tf + \mathbbm{1}_G \in Q + G(\varepsilon)$ for large enough $n$.

We will show that
\begin{equation} \label{eq:C_correctly_posed}
\lim_{\varepsilon \to 0+} u_\Gamma^{\sigma}(\Omega_+, \Omega_-;\varepsilon) = u_\Gamma^{\sigma}(\Omega_+, \Omega_-).
\end{equation}
From here it immediately follows that $(u_\Gamma^{\sigma}(\Omega_+, \Omega_-))_a = u_\Gamma^{\sigma}(\Omega_+, \Omega_-)$, hence the problem is well-posed and the strong duality $u_\Gamma^{\sigma}(\Omega_+, \Omega_-) = v_\Gamma^{\sigma}(\Omega_+, \Omega_-)$ holds.

It turns out that the question is easier to handle for the classical Delsarte constraint where we only have a restriction on the set of positivity.

\subsection{Classical Delsarte constraint}

We consider the case when $\Omega_- = G$, i.e. a restriction is given only on sets of positivity of functions $\varphi$. In this case $B = \emptyset$, the positive cone in the space $F = C(G)^{K}$ is $Q = M =\{ \ff \in C(G)^{K}~:~ \ff|_A \le 0\}$, and by Lemma~\ref{lemma:K-polar} the positive cone in the space $F^* = M(G)^{K}$ is $Q^* = M^* = \{ \mu \in M(G)^{K} : \mu|_{\overline{A}} \le 0, \ \mu|_{\overline{A}^c}\equiv 0 \}$.

The extremal problem $u_\Gamma^{\sigma}(\Omega_+, G)$ takes the form
\begin{align*}
u_\Gamma^{\sigma}(\Omega_+, G)  & =  \inf \left\{ \sum_{\gamma \in  \supp{f} \setminus \{\mathbbm{1}_G\}} f(\gamma) s(\gamma) : \right.   f \in \ell^1_+(\Gamma), \
\varphi|_A \le 0, \\ 
& \qquad\qquad\text{where }\left. \varphi = \mathbbm{1}_G + \sum_{\gamma \in  \supp{f} \setminus \{\mathbbm{1}_G\}} f(\gamma) \gamma
\right\} \\
& = \inf \left\{ \sum_{\gamma \in  \supp{f} \setminus \{\mathbbm{1}_G\}} f(\gamma) s(\gamma) : f \in \ell^1_+(\Gamma),  \ Tf + \mathbbm{1}_G \in M  \right\},
\end{align*}
and the extended problem $u_\Gamma^{\sigma}(\Omega_+, G;\varepsilon)$ is
\begin{align*}
u_\Gamma^{\sigma}(\Omega_+, G;\varepsilon) & = \inf \left\{ \sum_{\gamma \in  \supp{f} \setminus \{\mathbbm{1}_G\}} f(\gamma) s(\gamma) : f \in \ell^1_+(\Gamma), \
Tf + \mathbbm{1}_{G} \in M + G(\varepsilon)   \right\} \\
& = \inf \left\{ \sum_{\gamma \in  \supp{f} \setminus \{\mathbbm{1}_G\}} f(\gamma) s(\gamma) : f \in \ell^1_+(\Gamma), \
Tf + \mathbbm{1}_G \le \varepsilon \mathbbm{1}_{G} \ \text{on} \ A   \right\}.
\end{align*}

Arguing as Arestov and Babenko in \cite{AB}, we show that
\begin{equation} \label{eq:AB_A}
u_\Gamma^{\sigma}(\Omega_+, G;\varepsilon) = (1 - \varepsilon) u_\Gamma^{\sigma}(\Omega_+, G), \quad 0 < \varepsilon < 1.
\end{equation}
To prove \eqref{eq:AB_A} we notice that for each $g$ which is feasible in the problem $u_\Gamma^{\sigma}(\Omega_+, G;\varepsilon)$ (i.e., $g \in \ell^1_+(\Gamma)$ and $\sum_{\gamma \in \supp{g} \setminus \{\mathbbm{1}_G\}} g(\gamma) \gamma + (1 - \varepsilon) \mathbbm{1}_G \le 0$ on $A$),  the function $f = \frac{1}{1 - \varepsilon} g$ is feasible in the problem $u_\Gamma^{\sigma}(\Omega_+, G)$ (namely, $f \in \ell^1_+(\Gamma)$ and $\sum_{\gamma \in  \supp{f} \setminus \{\mathbbm{1}_G\}} f(\gamma) \gamma + \mathbbm{1}_G \le 0$ on $A$). Conversely, for each $f$ feasible in the problem $u_\Gamma^{\sigma}(\Omega_+, G)$, the function $g = (1 - \varepsilon) f$ is feasible in the problem $u_\Gamma^{\sigma}(\Omega_+, G;\varepsilon)$. Obviously,
$$
\sum_{\gamma \in \supp{g} \setminus \{\mathbbm{1}_G\}} g(\gamma) s(\gamma) = (1 - \varepsilon) \sum_{\gamma \in  \supp{f} \setminus \{\mathbbm{1}_G\}} f(\gamma) s(\gamma),
$$
and \eqref{eq:AB_A} follows. This immediately implies \eqref{eq:C_correctly_posed}, and we arrive at the following statement.

\begin{theorem} \label{theorem:one-sided}

Let $(G,K)$ be a compact Gelfand pair . Let $\Omega_+$ be a $K$-bi-invariant symmetric Borel subset of $G$ with $e \in \intt{\Omega_+}$. Let $\sigma \in M(G)^K$ be a positive definite measure satisfying Wiener's condition. Then
$$
u_\Gamma^{\sigma}(\Omega_+, G) = v_\Gamma^{\sigma}(\Omega_+, G).
$$

\end{theorem}

\subsection{The more general Delsarte-type constraint}

We now come back to the general case when restrictions are posed on both the set of positivity and the set of negativity of~$\varphi$.

Recall that
$$
\mathcal{A}^{K,\sigma}_G(\Omega_+, \Omega_{-}) = \sup \left\{\int_G \varphi \, \mbox{d}\lambda_G : \varphi \in \mathcal{F}^{K,\sigma}_G(\Omega_+, \Omega_-)  \right\},
$$
where
$$
\mathcal{F}^{K,\sigma}_G(\Omega_+, \Omega_-) = \{ \varphi \in C(G)^K : \varphi  \gg 0, \ \langle \ff, \sigma \rangle_2 = 1, \  \varphi|_{\Omega_+^c} \le 0, \  \varphi|_{\Omega_-^c} \ge 0 \}.
$$
To establish the duality result, we consider the problem on the extended class
\begin{align*}
\mathcal{F}^{K,\sigma}_G(\Omega_+, \Omega_-; \varepsilon) & := \{ \varphi \in C(G)^K : \varphi  \gg 0, \ \langle \ff, \sigma \rangle_2 = 1, \  \varphi \in Q + G(\varepsilon) \} \\
& = \{ \varphi \in C(G)^K : \varphi  \gg 0, \ \langle \ff, \sigma \rangle_2 = 1, \  \varphi|_{\Omega_+^c} \le \varepsilon, \  \varphi|_{\Omega_-^c} \ge -\varepsilon \},
\end{align*}
and the extremal value 
$$
\mathcal{A}^{K,\sigma}_G(\Omega_+, \Omega_-;\varepsilon) := \sup\left\{\int_G \varphi \, \mbox{d}\lambda_G  : \varphi \in \mathcal{F}^{K,\sigma}_G(\Omega_+, \Omega_-; \varepsilon) \right\}, 
$$
where $\varepsilon > 0$. 

We will employ Assumption~O that was introduced in Section~\ref{section:introduction}.

\begin{lemma} \label{lemma:C_epsilon}

Let $(G,K)$ be a compact Gelfand pair. Let $\Omega_+$, $\Omega_-$ be $K$-bi-invariant symmetric Borel subsets of $G$ with $e \in \intt{\Omega_+}$, and  $\Omega_+^c$, $\Omega_-^c$ satisfying Assumption~O. 
Assume that $\sigma\in M(G)^K$ has the form $\sigma = \delta_e^K + \tau $, where \(\tau \in M(G)^K\) is positive definite and absolutely continuous with respect to the Haar measure $\lambda_G$.
Then
$$
\lim_{\varepsilon \to 0+} \mathcal{A}^{K,\sigma}_G(\Omega_+, \Omega_-;\varepsilon) = \mathcal{A}^{K,\sigma}_G(\Omega_+, \Omega_{-}).
$$

\end{lemma}

\begin{proof}

Since $\widehat{\delta_e^K} = \mathbbm{1}_\Gamma$, the measure $\sigma$ is positive definite and satisfies Wiener's condition. 

The proof uses ideas from the papers \cite{BRR_Expo} by Berdysheva, Ramabulana, R\'ev\'esz, \cite{marcell-zsuzsa} by Ga\'al, Nagy-Csiha, and  \cite{ramabulana} by Ramabulana.

It is clear that 
\begin{equation} \label{eq:eps_lemma_C_class_incl}
\mathcal{F}^{K,\sigma}_G(\Omega_+, \Omega_-) \subset \mathcal{F}^{K,\sigma}_G(\Omega_+, \Omega_-; \varepsilon_1) \subset \mathcal{F}^{K,\sigma}_G(\Omega_+, \Omega_-; \varepsilon_2), \quad 0 < \varepsilon_1 < \varepsilon_2.
\end{equation}
Consequently,
$$
\mathcal{A}^{K,\sigma}_G(\Omega_+, \Omega_{-}) \le \mathcal{A}^{K,\sigma}_G(\Omega_+, \Omega_-;\varepsilon_1) \le \mathcal{A}^{K,\sigma}_G(\Omega_+, \Omega_-;\varepsilon_2), \quad 0 < \varepsilon_1 < \varepsilon_2.
$$
Since the quantity $\mathcal{A}^{K,\sigma}_G(\Omega_+, \Omega_-;\varepsilon)$ decreases as $\varepsilon$ monotonically decreases to $0$, and is bounded below by $\mathcal{A}^{K,\sigma}_G(\Omega_+, \Omega_-)$, the limit
$$
\widetilde{\mathcal{A}} := \lim_{\varepsilon \to 0+} \mathcal{A}^{K,\sigma}_G(\Omega_+, \Omega_-;\varepsilon) 
$$
exists and
\begin{equation} \label{eq:A-positive}
\widetilde{\mathcal{A}} \ge \mathcal{A}^{K,\sigma}_G(\Omega_+, \Omega_{-}) > 0.
\end{equation}
We wish to show that
$$
\widetilde{\mathcal{A}} = \mathcal{A}^{K,\sigma}_G(\Omega_+, \Omega_{-}).
$$
Denote $\mathcal{A}_n := \mathcal{A}^{K,\sigma}_G \left( \Omega_+, \Omega_-;\frac{1}{n} \right)$. 
Clearly, $\widetilde{\mathcal{A}} = \lim_{n \to \infty} \mathcal{A}_n$.

Take $\psi_n \in \mathcal{F}^{K,\sigma}_G \left( \Omega_+, \Omega_-; \frac{1}{n} \right)$ such that 
\begin{equation} \label{eq:eps_lemma_C_choice_gn}
\mathcal{A}_n - \frac{1}{n} \le \int_G  \psi_n \,\mbox{d}\lambda_G\le \mathcal{A}_n.
\end{equation}
Since $\langle \psi_n, \tau \rangle_2 \ge 0$, we have
$$
1 = \langle \psi_n, \sigma \rangle_2 = \psi_n(e) + \langle \psi_n, \tau \rangle_2 \ge \psi_n(e),
$$
and thus $\|\psi_n\|_\infty \le 1$. Invoking the fact that $\lambda_G(G) = 1$, we have
$$
\int_G |\psi_n|^2 \,\mbox{d}\lambda_G \le \|\psi_n\|_\infty^2 \lambda_G(G) = 1.
$$
Therefore, the sequence $(\psi_n)_{n \in \NN}$ belongs to the closed unit ball of the space $L^2(G)$ which is weakly sequentially compact (e.g.~\cite[Theorem 3.18]{Brezis}). Thus, there is a subsequence of $(\psi_n)_{n \in \NN}$  that converges weakly in $L^2(G)$ to some function $\psi \in L^2(G)$. For simplicity we assume that $(\psi_n)_{n \in \NN}$ itself is such a sequence.

Note that each class $\mathcal{F}^{K,\sigma}_G \left( \Omega_+, \Omega_-; \frac{1}{n} \right)$ is convex, and~\eqref{eq:eps_lemma_C_class_incl} holds.
By Mazur's Lemma (e.g.~\cite[Corollary 3.8, Exercise 3.4]{Brezis}), there exists a sequence $(\Psi_n)_{n \in \NN}$ with $\Psi_n \in \conv{ \left( \bigcup_{k=n}^\infty \{\psi_k\} \right) }$ that converges to $\psi$ strongly in $L^2(G)$.  Take $\varepsilon > 0$. Since $\widetilde{\mathcal{A}} = \lim_{n \to \infty} \mathcal{A}_n$ and $(\mathcal{A}_n)_{n \in \NN}$ is monotonically decreasing, there exists $N \in \NN$ such that
$$
\widetilde{\mathcal{A}} \le \mathcal{A}_k \le \widetilde{\mathcal{A}} + \varepsilon \quad \text{for all } k \ge N.
$$
Assume that  $n \ge N$. Taking into account \eqref{eq:eps_lemma_C_choice_gn}, we obtain for each $k \ge n$
$$
\widetilde{\mathcal{A}} - \frac{1}{n} \le \mathcal{A}_k - \frac{1}{k} \le \int_G  \psi_k \,\mbox{d}\lambda_G  \le \mathcal{A}_k \le \widetilde{\mathcal{A}} + \varepsilon.
$$
Each $\Psi_n$ has the form $\Psi_n = \sum_{k =n}^\infty \alpha_k^{(n)} \psi_k$, where $\alpha_k^{(n)} \ge 0$, $\sum_{k =n}^\infty \alpha_k^{(n)} = 1$, and $\alpha_k^{(n)} = 0$ for all but finitely many~$k$. By linearity we have
$$
\widetilde{\mathcal{A}} - \frac{1}{n} \le \int_G \Psi_n \,\mbox{d}\lambda_G \le \widetilde{\mathcal{A}} + \varepsilon,
$$
and consequently
$$
\widetilde{\mathcal{A}} 
\le \liminf_{n \to \infty} \int_G  \Psi_n \,\mbox{d}\lambda_G 
\le \limsup_{n \to \infty} \int_G  \Psi_n \,\mbox{d}\lambda_G
\le \widetilde{\mathcal{A}} + \varepsilon.
$$
Since this is true for any $\varepsilon > 0$, we conclude that
$$
%\begin{equation} \label{eq:eps_lemma_C-limit}
\lim_{n \to \infty} \int_G  \Psi_n \,\mbox{d}\lambda_G = \widetilde{\mathcal{A}}.
%\end{equation}
$$

By the inclusion~\eqref{eq:eps_lemma_C_class_incl} and the convexity of the classes $\mathcal{F}^{K,\sigma}_G \left( \Omega_+, \Omega_-; \frac{1}{n} \right)$, we have $\Psi_n \in \mathcal{F}^{K,\sigma}_G \left( \Omega_+, \Omega_-; \frac{1}{n} \right)$. The sequence $(\Psi_n)_{n \in \NN}$ converges to $\psi$ strongly and weakly in~$L^2(G)$. The strong convergence in $L^2(G)$ implies that $(\Psi_n)_{n \in \NN}$ contains a subsequence that converges to $\psi$ pointwise almost everywhere with respect to the Haar measure $\lambda_G$ (e.g.,~\cite[Theorem~4.9]{Brezis}). For simplicity we assume that $(\Psi_n)_{n \in \NN}$ itself is such a sequence.

Since $(\Psi_n)_{n \in \NN}$ converges to $\psi$ weakly in~$L^2(G)$, we have that
$$
\int_G \psi (\xi \star \widetilde{\xi}) \,\mbox{d}\lambda_G =  \lim_{n \to \infty} \int_G \Psi_n (\xi \star \widetilde{\xi}) \,\mbox{d}\lambda_G \ge 0
$$
for any complex-valued function $\xi \in C(G)$. This means that $\psi$ is an integrally positive definite function. Since $G$ is compact (and in particular $\sigma$-compact), $\psi$ agrees almost everywhere with respect to $\lambda_G$ with a continuous positive definite function (e.g.~\cite[Theorem 1.7.3]{sasvari}). For simplicity we denote this continuous function again by $\psi$.  The sequence $(\Psi_n)_{n \in \NN}$ converges to the continuous positive definite function $\psi$ strongly and weakly in $L^2(G)$ and pointwise almost everywhere with respect to $\lambda_G$.

Now we are going to exploit the fact that the sets $A = \Omega_+^c$, $B = \Omega_-^c$ satisfy Assumption O. Since $\Psi_n \in \mathcal{F}^{K,\sigma}_G \left( \Omega_+, \Omega_-; \frac{1}{n} \right)$ we have $\Psi_n|_{A} \le \frac{1}{n}$.  If $g \in A$ is a point where $\lim_{n \to \infty} \Psi_n(g) = \psi(g)$, then $\psi(g) \le 0$ follows immediately. Thus, $\psi \le 0$ holds almost everywhere on~$A$. We want to show that $\psi|_{A} \le 0$. Suppose, by contradiction, that there is $g_0 \in A$ with $\psi(g_0) > 0$. Since $\psi$ is continuous, there exists an open neighborhood $V$ of~$g_0$ such that $\psi(g) > 0$ for all $g \in V$. By Assumption O, the set $A \cap V \subset A$ has a positive measure $\lambda_G(A \cap V) > 0$ and $\psi(g) > 0$ for any $g \in A \cap V$, which is a contradiction. This proves that $\psi|_{A} \le 0$. Similarly, $\psi|_{B} \ge 0$.

Since the constant function $\mathbbm{1}_G$ is in $L^2(G)$ on the account of $G$ being compact, by the weak convergence we have that 
$$
%\begin{equation} \label{eq:eps_lemma_C-limit-int}
\int_G \psi \,\mbox{d}\lambda_G = \int_G \psi \mathbbm{1}_G \,\mbox{d}\lambda_G 
= \lim_{n \to \infty} \int_G \Psi_n \mathbbm{1}_G \,\mbox{d}\lambda_G 
=  \lim_{n \to \infty} \int_G \Psi_n \,\mbox{d}\lambda_G = \widetilde{\mathcal{A}}.
%\end{equation}
$$

Consider the $K$-periodisation $\psi^K$ of $\psi$. By Lemma \ref{Ksymmetrisation}, we have that $\psi^K \in C(G)^K$, $\psi^K \gg 0$, $\psi^{K}|_{\Omega_+^c} \le 0$, $\psi^{K}|_{\Omega_-^c} \ge 0$, and 
\begin{equation} \label{eq:eps_lemma_C-limit-int}
\int_{G}\psi^{K} \,\mbox{d}\lambda_G = \int_{G}\psi \,\mbox{d}\lambda_{G} = \widetilde{\mathcal{A}} > 0.
\end{equation}
This implies, in particular, that $\psi^K \not\equiv 0$, and hence $\langle \psi^K,\sigma \rangle_2 > 0$. The function $\psi^K$ satisfies all requirements in the definition of the class $\widetilde{\mathcal{F}}^{K,\sigma}_G(\Omega_+, \Omega_-)$. This implies, by the definition of the extremal constant, \eqref{eq:A-positive} and \eqref{eq:eps_lemma_C-limit-int}, that the following estimate holds:
\begin{equation}\label{eq:lim1}
\frac{\int_G \psi^K \,\mbox{d}\lambda_G}{\langle \psi^K, \sigma \rangle_2} 
\le \mathcal{A}^{K,\sigma}_G(\Omega_+, \Omega_{-})
\le \widetilde{\mathcal{A}} 
= \int_{G}\psi^{K} \,\mbox{d}\lambda_G.
\end{equation}
We conclude that
\begin{equation} \label{eq:eps_lemma_C-normalize}
\langle \psi^K, \sigma \rangle_2 \ge 1.
\end{equation}
Now we use the facts that $\tau$ is $K$-bi-invariant and absolutely continuous with respect to the Haar measure $\lambda_G$.  In particular, the conditions $|\Psi_n| \le \mathbbm{1}_G$ and $\Psi_n \to \psi$ hold also almost everywhere with respect to \(\tau\). Moreover, $\mathbbm{1}_G$ is integrable with respect to $\tau$. Thus, by the Lebesgue Dominated Convergence Theorem
$$
\lim_{n \to \infty} \langle \Psi_n, \tau \rangle_2 = \lim_{n \to \infty} \int_G \Psi_n \,\mbox{d}\tau 
= \int_G \psi \, \mbox{d}\tau = \langle \psi^K, \tau \rangle_2.
$$
Together with $\lim_{n \to \infty} \langle \Psi_n, \sigma \rangle_2 = 1$ this implies that the limit $\lim_{n \to \infty} \Psi_n(e)$ exists and 
$\lim_{n \to \infty} \Psi_n(e) = 1 - \langle \psi^K, \tau \rangle_2$.
The almost everywhere pointwise convergence of $(\Psi_n)_{n \in \NN}$ to $\psi$ implies that 
$$
\|\psi\|_\infty \le \lim_{n \to \infty} \| \Psi_n \|_\infty = \lim_{n \to \infty} \Psi_n(e) = 1 - \langle \psi^K, \tau \rangle_2.
$$
Now we have
$$
\langle \psi^K, \delta_e^K \rangle_2 = \int_{K}\int_{K}\psi(kek') \,\mbox{d}\lambda_K(k) \,\mbox{d}\lambda_{K}(k') 
\le \int_{K}\int_{K} \| \psi\|_\infty \,\mbox{d}\lambda_{K}(k) \,\mbox{d}\lambda_{K}(k') \le 1 - \langle \psi^K, \tau \rangle_2,
$$
which yields
$$
\langle \psi^K, \sigma \rangle_2 = \langle \psi^K, \delta_e^K \rangle_2 + \langle \psi^K, \tau \rangle_2 \le 1.
$$
Together with~\eqref{eq:eps_lemma_C-normalize}, this gives
$$
\langle \psi^K, \sigma \rangle_2 = 1.
$$
Thus, all quantities in~\eqref{eq:lim1} are equal, and therefore
$$
\mathcal{A}^{K,\sigma}_G(\Omega_+, \Omega_{-}) = \widetilde{\mathcal{A}} =  \lim_{\varepsilon \to 0+} \mathcal{A}^{K,\sigma}_G(\Omega_+, \Omega_-;\varepsilon). 
$$
This proves the lemma. 
\end{proof}

Lemma \ref{lemma:C_epsilon} gives the following strong duality statement.

\begin{theorem}\label{main_theorem}

Let $(G,K)$ be a compact Gelfand pair. Let $\Omega_+$, $\Omega_-$ be $K$-bi-invariant symmetric Borel subsets of $G$ with $e \in \intt{\Omega_+}$, and $\Omega_+^c$, $\Omega_-^c$ satisfying Assumption O. Assume that $\sigma\in M(G)^K$ has the form $\sigma = \delta_e^K + \tau $, where \(\tau \in M(G)^K\) is positive definite and absolutely continuous with respect to the Haar measure $\lambda_G$. Then
$$
u_\Gamma^{\sigma}(\Omega_+, \Omega_-) = v_\Gamma^{\sigma}(\Omega_+, \Omega_-).
$$

\end{theorem}

\begin{proof}

To prove the theorem, we must show \eqref{eq:C_correctly_posed}.

The Delsarte-type problem equivalent to $u_\Gamma^{\sigma}(\Omega_+, \Omega_-;\varepsilon)$ differs from $\mathcal{A}^{K,\sigma}_G(\Omega_+, \Omega_-;\varepsilon)$ by the normalisation: instead of $\langle \ff, \sigma \rangle_2 = 1$, we need the normalisation $\int_G \ff \,\mbox{d}\lambda_G =1$. Therefore, for the sake of this proof we introduce the extremal constant, with the notation $s = \widehat{\sigma}$,
\begin{align*}
\mathcal{A}_G^{K,\sigma,1}(\Omega_+, \Omega_-; \varepsilon) & := 
\sup \left\{ \frac{\int_G \ff \,\mbox{d} \lambda_G}{\langle \ff, \sigma \rangle_2} : \ff \in C(G)^K, \ \ff \gg 0, \right.\\
& \left.
\hspace{14mm} \int_G \ff \,\mbox{d}\lambda_G =1, \ \ff|_{\Omega_+^c} \le \varepsilon, \ \ff|_{\Omega_-^c} \ge -\varepsilon
\right\} 
 \\
& = \sup \left\{ \frac{ 1 }{ s(\mathbbm{1}_G) + \sum_{\gamma \in \supp{f} \setminus \{\mathbbm{1}_G\}} f(\gamma) s(\gamma) }  : \right.   f \in \ell^1_+(\Gamma), \
\varphi|_A \le \varepsilon, \\ 
& \hspace{14mm} \varphi|_B \ge -\varepsilon, \text{ where }\left. \varphi = \mathbbm{1}_G + \sum_{\gamma \in \supp{f} \setminus \{\mathbbm{1}_G\}} f(\gamma) \gamma
\right\}.
\end{align*}
We have
$$
\mathcal{A}_G^{K,\sigma,1}(\Omega_+, \Omega_-;\varepsilon) = \frac{1}{s(\mathbbm{1}_G) + u_\Gamma^{\sigma}(\Omega_+, \Omega_-;\varepsilon)}.
$$
Note that admissible functions $\varphi$ in the problem $\mathcal{A}_G^{K,\sigma,1}(\Omega_+, \Omega_-;\varepsilon)$ satisfy $\int_G \varphi \,\mbox{d}\lambda_G = 1$. Since  $\int_G \varphi \,\mbox{d}\lambda_G \le \varphi(e) \lambda_G(G) = \varphi(e)$, we also have $\varphi(e) \ge 1$, and therefore $\langle \ff, \sigma \rangle_2 \ge 1$. Now take $\psi := \frac{\varphi}{\langle \ff, \sigma \rangle_2}$. We have $\psi \in C(G)^K$, $\psi \gg 0$, $\langle \psi, \sigma \rangle_2 = 1$, $\psi|_A \le \frac{\varepsilon}{\langle \ff, \sigma \rangle_2} \le \varepsilon$, $\psi|_B \ge -\frac{\varepsilon}{\langle \ff, \sigma \rangle_2} \ge -\varepsilon$. 
Thus, $\psi \in \mathcal{F}^{K,\sigma}_G(\Omega_+, \Omega_-; \varepsilon)$, and
$$
\frac{\int_G \ff \, \mbox{d}\lambda_G}{\langle \ff, \sigma \rangle_2} = \int_G \psi \, \mbox{d}\lambda_G.
$$
Hence,
$$
\mathcal{A}_G^{K,\sigma,1}(\Omega_+, \Omega_-;\varepsilon) \le \mathcal{A}^{K,\sigma}_G(\Omega_+, \Omega_-;\varepsilon).
$$
Since $u_\Gamma^{\sigma}(\Omega_+, \Omega_-;\varepsilon)$ is an extension of the minimization problem $u_\Gamma^\sigma(\Omega_+,\Omega_-)$, we have that $u_\Gamma^{\sigma}(\Omega_+, \Omega_-) \ge u_\Gamma^{\sigma}(\Omega_+, \Omega_-;\varepsilon)$. Thus,
\begin{align*}
\mathcal{A}^{K,\sigma}_G(\Omega_+, \Omega_{-}) & = \frac{1}{s(\mathbbm{1}_G) + u_\Gamma^{\sigma}(\Omega_+, \Omega_-)} 
\le \frac{1}{s(\mathbbm{1}_G) + u_\Gamma^{\sigma}(\Omega_+, \Omega_-;\varepsilon)} \\
& = \mathcal{A}_G^{K,\sigma,1}(\Omega_+, \Omega_-;\varepsilon) \le \mathcal{A}^{K,\sigma}_G(\Omega_+, \Omega_-;\varepsilon).
\end{align*}
By Lemma~\ref{lemma:C_epsilon},
$\lim_{\varepsilon \to 0+} \mathcal{A}^{K,\sigma}_G(\Omega_+, \Omega_-;\varepsilon) = \mathcal{A}^{K,\sigma}_G(\Omega_+, \Omega_{-})$,
which in turn implies~\eqref{eq:C_correctly_posed}. We conclude the problem $u_\Gamma^{\sigma}(\Omega_+, \Omega_-)$ is well-posed. This implies that the strong duality relation
$$
u_\Gamma^{\sigma}(\Omega_+, \Omega_-) = v_\Gamma^{\sigma}(\Omega_+, \Omega_-)
$$
holds true.

\end{proof}

Now we are going to deduce from Theorems~\ref{theorem:one-sided} and~\ref{main_theorem} a generalisation of Theorem~\ref{theorem:main_group} from Section~\ref{section:introduction}. We introduce the class of measures
$$
T_0(G)^K := \{ \tau \in M(G)^K : \tau \gg 0, \ \tau(G)=\widehat{\tau}(\mathbbm{1}_G) = 0 \}.
$$

\begin{lemma} \label{lemma:dual_formulation}
Let $\sigma \in M(G)^K$ be a positive definite measure. For the extremal value 
\begin{equation} \label{eq:problem_v_copy}
v_\Gamma^{\sigma}(\Omega_+, \Omega_-) 
= \sup \left\{ - \mu(G) : \mu \in Q^*, \ \widehat\mu(\gamma) \le \widehat\sigma(\gamma) \text{ for all } \gamma \ne \mathbbm{1}_G\right\}
\end{equation}
we have
\begin{equation} \label{eq:v_positive}
v_\Gamma^\sigma(\Omega_+,\Omega_-) \ge 0
\end{equation}
and
\begin{equation} \label{eq:lemma_vv}
v_\Gamma^\sigma(\Omega_+,\Omega_-) + \sigma(G) = \sup\{ z \in \RR : \exists \ \mu\in Q^* \text{ and } \tau \in T_0(G)^K
 \text{ with } \sigma = z\lambda_G + \mu + \tau\}. 
\end{equation}
\end{lemma}

\begin{proof}
Note that the measure $\mu \equiv 0$ is admissible in~\eqref{eq:problem_v_copy}. Therefore, \eqref{eq:v_positive} holds.

We denote the right-hand side of~\eqref{eq:lemma_vv} by $\mathcal{Z}$.
Suppose that $z \in \RR$ is such that
\[
\sigma = z\lambda_G+\mu+\tau \quad \text{with some } \mu\in Q^*, \ \tau \in T_0(G)^K.
\]
Taking Fourier transforms at $\gamma\neq \mathbbm{1}_G$ gives
\[
\widehat{\sigma}(\gamma) = \widehat{\mu}(\gamma) + \widehat{\tau}(\gamma) \ge \widehat{\mu}(\gamma).
\]
Thus, $\mu$ is feasible in~\eqref{eq:problem_v_copy}.
Evaluating at $\mathbbm{1}_G$ yields
\[
\widehat{\sigma}(\mathbbm{1}_G) = \sigma(G) = z + \mu(G),
\]
so
\[
z = \sigma(G) - \mu(G) \le \sigma(G) + v_\Gamma^{\sigma}(\Omega_+, \Omega_-).
\]
It follows that
$$
\mathcal{Z} \le \sigma(G) + v_\Gamma^{\sigma}(\Omega_+, \Omega_-).
$$

Conversely, assume that $\mu\in Q^*$ satisfies  $\widehat{\mu}(\gamma)\le \widehat{\sigma}(\gamma)$ for all $\gamma\neq \mathbbm{1}_G$. Take $z = \sigma(G) - \mu(G)$ and the measure $\tau$ defined by
\[
\sigma = z\lambda_G + \mu + \tau.
\]
Clearly, $\tau \in M(G)^K$. Its Fourier transform is $\widehat{\tau}(\mathbbm{1}_G) = 0$ and $\widehat{\tau}(\gamma) = \widehat{\sigma}(\gamma) - \widehat{\mu}(\gamma) \ge 0$ for $\gamma\neq \mathbbm{1}_G$. 
Thus, $\tau \in T_0(G)^K$, and so $z = \sigma(G) - \mu(G)$ is feasible in the problem $\mathcal{Z}$. This gives
$$
\sigma(G) - \mu(G) = z \le \mathcal{Z},
$$
and consequently 
$$
\sigma (G) + v_\Gamma^{\sigma}(\Omega_+, \Omega_-) \le \mathcal{Z}.
$$
This proves~\eqref{eq:lemma_vv}.
\end{proof}

Now we are going to formulate the main result of this section. We consider the linear programming problems
\begin{equation*}
{\mathcal{A}}^{K, \sigma}_G(\Omega_+, \Omega_-) = \sup \left\{ \int_G\varphi\,\mathrm{d}\lambda_G : \varphi \in \mathcal{F}^{K,\sigma}_G(\Omega_+, \Omega_-) \right\}
\end{equation*}
and
\begin{align*}
{\mathcal{A}}^{K, \sigma}_G(\Omega_+, \Omega_-)^{*}  
& := \inf \{ \alpha \in \RR : \alpha \sigma - \lambda_G \in Q^* + T_0(G)^K \}.
\end{align*} 

\begin{theorem} \label{theorem:main_duality_G}
Let $(G,K)$ be a compact Gelfand pair, and let $\Omega_+$ be a $K$-bi-invariant symmetric Borel subset of $G$ with $e \in \intt{\Omega_+}$. Let $\sigma \in M(G)^K$ be a positive definite measure that satisfies Wiener's condition. Then the problem ${\mathcal{A}}^{K, \sigma}_G(\Omega_+, G)^{*}$ is the dual problem of ${\mathcal{A}}^{K, \sigma}_G(\Omega_+, G)$, and we have the strong duality relation
\begin{equation*}
   {\mathcal{A}}^{K, \sigma}_G(\Omega_+, G) = {\mathcal{A}}^{K, \sigma}_G(\Omega_+, G)^{*}.
\end{equation*}
\end{theorem}

\begin{theorem} \label{theorem:main_duality}
Let $(G,K)$ be a compact Gelfand pair, and let $\Omega_+$, $\Omega_-$ be $K$-bi-invariant symmetric Borel subsets of $G$ with $e \in \intt{\Omega_+}$ satisfying Assumption~O.  Assume that $\sigma\in M(G)^K$ has the form $\sigma = \delta_e^K + \tau $, where \(\tau \in M(G)^K\) is positive definite and absolutely continuous with respect to the Haar measure $\lambda_G$. Then the problem ${\mathcal{A}}^{K, \sigma}_G(\Omega_+, \Omega_-)^{*}$ is the dual problem of ${\mathcal{A}}^{K, \sigma}_G(\Omega_+, \Omega_-)$, and we have the strong duality relation
\begin{equation*}
   {\mathcal{A}}^{K, \sigma}_G(\Omega_+, \Omega_-) = {\mathcal{A}}^{K, \sigma}_G(\Omega_+, \Omega_-)^{*}.
\end{equation*}
\end{theorem}

\begin{proof} 
Let $\mathcal{Z}$ again denote the right-hand side of~\eqref{eq:lemma_vv}. Equations~\eqref{eq:v_positive} and~\eqref{eq:lemma_vv} imply $\mathcal{Z} \ge \sigma(G) > 0$. The substitution $z = \frac{1}{\alpha}$ gives
\begin{align*}
\mathcal{Z}
& = \sup\{ z > 0 : \sigma - z\lambda_G \in Q^* + T_0(G)^K \} \\
& = \sup\left\{ \frac{1}{\alpha} > 0 : \alpha \sigma - \lambda_G \in Q^* + T_0(G)^K \right\}.
\end{align*}
It follows that
$$
{\mathcal{A}}^{K, \sigma}_G(\Omega_+, \Omega_-)^{*} = \frac{1}{\mathcal{Z}}.
$$
Now using~\eqref{eq:C-u}, Theorem~\ref{theorem:one-sided} for Theorem~\ref{theorem:main_duality_G} and Theorem~\ref{main_theorem} for Theorem~\ref{theorem:main_duality}, respectively, and Lemma~\ref{lemma:dual_formulation}, we obtain
$$
{\mathcal{A}}^{K, \sigma}_G(\Omega_+, \Omega_-) = \frac{1}{\sigma(G) + u_\Gamma^\sigma(\Omega_+, \Omega_-)}
= \frac{1}{\sigma(G) + v_\Gamma^\sigma(\Omega_+, \Omega_-)} = \frac{1}{\mathcal{Z}} = {\mathcal{A}}^{K, \sigma}_G(\Omega_+, \Omega_-)^{*}.
$$

\end{proof}

In conclusion, let us see how Theorems~\ref{theorem:main_group} and~\ref{theorem:main_group_G} follow from what has been done above. They are special cases of the more general Theorems~\ref{theorem:main_duality} and~\ref{theorem:main_duality_G}, when $K=\{0\}$, $\sigma = \delta_0^{\{0\}} = \delta_0$.

\section{Acknowledgements}

This research was partially supported by the DAAD-Tempus PPP Grant 57448965 ``Harmonic Analysis and Extremal Problems''.

Elena E. Berdysheva was supported  in part by the University of Cape Town's Research Committee (URC).

Elena E. Berdysheva and Mita D. Ramabulana thank the HUN-REN R\'enyi Institute of Mathematics for hospitality during their respective visits.

Marcell Ga\'al was supported by the National Research, Development and Innovation Office -- NKFIH Reg. No.'s K-115383 and K-128972, and also by the Ministry for Innovation and Technology, Hungary throughout Grant TUDFO/47138-1/2019-ITM.

Mita D. Ramabulana was supported by the Carnegie DEAL 3 Postdoctoral Fellowship.

Szil\'ard Gy. Révész was supported in part by the Hungarian National Research, Development and Innovation Fund projects \# K-119528, K-132097, K-146387, K-147153 and Excellence No.
151341.

%%%%%%%%%%%%%%%%%%%%%%%%%%%%%%%%%%%%%%%%%%%%%%%%%%%%%%%%
%%%%%%%%%%%%%%%%%%%%%%%%%%%%%%%%%%%%%%%%%%%%%%%%%%%%%%%%

%\newpage

\bibliographystyle{amsplain}

\begin{thebibliography}{99}

\bibitem{AB}
V.\,V.~Arestov and A.\,G.~Babenko, \emph{On the Delsarte scheme for estimating contact numbers},
Proc. Steklov Inst. Math. \textbf{4} (1997), 36--65.

\bibitem{areberd} 
V.\,V.~Arestov and E.\,E.~Berdysheva, \emph{Tur\'an's problem for positive definite functions with supports in a hexagon}, Proc. Steklov Inst. Math., Supp 1. (2001), 20--29.

\bibitem{areberd0} 
V.\,V.~Arestov and E.\,E.~Berdysheva, \emph{Tur\'an's problem for a class of polytopes}, East J. Approx. (2002), no.~8, 381--388.

\bibitem{C-paper} 
E.\,E.~Berdysheva, B.~Farkas, M.\,D.~Ramabulana, M.~Ga\'al, and Sz.\,Gy.~R\'ev\'esz, \emph{Duality for Delsarte's extremal problem on locally compact Abelian groups}, ArXiv preprint, arXiv:2603.18287.

\bibitem{Berdysheva}
E.\,E.~Berdysheva and Sz.\,Gy.~R\'ev\'esz,
\emph{Delsarte's extremal problem and packing on locally compact Abelian groups},  Annali della Scuola Normale di Pisa – Classe di Scienze (5), {\bf XXIV} (2023), 1007--1052. 

\bibitem{BRR_Expo}
E.\,E.~Berdysheva, M.\,D.~Ramabulana, and Sz.\,Gy.~Révész, \emph{On extremal problems of Delsarte type for positive definite functions on LCA groups}, Expo. Math. {\bf 44} (2026), 125663.

\bibitem{bergetal}
C.~Berg, A.\,P.~Peron, and E.~Porcu, \emph{Orthogonal expansions related to compact Gelfand pairs}, Expo. Math. {\bf 36} (2018), 259--277.

\bibitem{Brezis} 
H.~Brezis, \emph{Functional Analysis, Sobolev Spaces, and Partial Differential Equations}, Springer-Verlag, New York (2010).

\bibitem{frederik}
F.~Broucke and Sz.\,Gy.~Révész, \emph{On Landau's local method obtaining zero-free regions}, manuscript.


\bibitem{cohn-measure-theory}
D.\,L.~Cohn, \emph{Measure Theory}, 2nd edition, Birkh\"auser, 2013.

\bibitem{cohn-elkies}
H.~Cohn and N.~Elkies, \emph{New upper bounds for sphere packings, I}, Ann. of Math. (2) \textbf{157} (2003), 689--714.

\bibitem{Cohn-Laat-Salmon} 
H.~Cohn, D.~de Laat, A.~Salmon, \emph{Three-point bounds for sphere packing},  ArXiv prerint, arXiv:2206.15373.

\bibitem{Cohn} 
H.~Cohn, \emph{New upper bounds on sphere packings. II}, Geom. Topol. \textbf{6} (2002) 329--353.

\bibitem{CKMRV} 
H.~Cohn, A.~Kumar, S.\,D.~Miller, D.~Radchenko, and M.~Viazovska, \emph{The sphere packing problem in dimension 24}, Ann. of Math. \textbf{185} (2017) 1017--1033.

\bibitem{cohn-zhao} 
H.~Cohn and Y.~Zhao, \emph{Sphere packing via spherical codes}, Duke Math. J. {\bf 163} (2014), 1965--2002.

\bibitem{delsarte1} 
P.~Delsarte, \emph{Bounds for unrestricted codes by linear programming}, Philips Res. Rep. \textbf{2} (1972), 272--289.

\bibitem{delsarte2} 
P.~Delsarte, J.-M.~Goethals, and J.\,J.~Seidel, \emph{Spherical codes and designs, Geom. Dedicata} (3) \textbf{6} (1977), 363–388.

\bibitem{Duffin}
R.\,J.~Duffin, \emph{Infinite programs}, In: Linear Inequalities and Related Systems (AM-38), ed. by H.\,W.~Kuhn and A.\,W.~Tucker, Princeton University Press, 1956, pp. 157--70.

\bibitem{edwards}
R.\,E.~Edwards, \emph{Functional Analysis: Theory and Applications}, Reinhart and Winston, New York, 1965.

\bibitem{marcell-zsuzsa}
M.~Ga\'al and Zs.~Nagy-Csiha,
\emph{On the existence of an extremal function in the Delsarte extremal problem},
Mediterr. J. Math. \textbf{17} (2020), no.~190.

\bibitem{gabardo}
J.-P.~Gabardo,
\emph{The Turán Problem and Its Dual for Positive Definite Functions Supported on a Ball in $\mathbb{R}^d$},
J. Fourier Anal. Appl. \textbf{30} (2024), 11.

\bibitem{Gaspari-Cohn}
G.~Gaspari and S.\,E.~Cohn, \emph{Construction of correlation functions in two and three dimensions}, Q. J. R. Meteorol. Soc. \textbf{125} (1999), 723--757.

\bibitem{gneiting} 
T.~Gneiting, \emph{Supplement to Strictly and non-strictly positive definite functions on spheres} (2013), doi: DOI:10.3150/12-BEJSP06SUPP.

\bibitem{golstein}
E.\,G.~Gol'stein, \emph{Teoriya dvoistvennosti v matematicheskom programmirovanii i eyo prilozheniya}
(Theory of Duality in Mathematical Programming and Its Applications), Moscow: Nauka, 1971.

\bibitem{GorbachevTula} 
D.\,V.~Gorbachev, \emph{Extremal problems for entire functions of exponential spherical type, connected with the Levenshtein bound on the sphere packing density in $\mathbb{R}^n$} (Russian), Izvestiya of the Tula State University, Ser. Mathematics, Mechanics, Informatics \textbf{6} (2000) 71--78.

\bibitem{gorbachev}
D.\,V.~Gorbachev,
\emph{An extremal problem for periodic functions with supports in the ball},
Math. Notes \textbf{3} (2001), 313--319.

\bibitem{gorbachevduality}
D.\,V.~Gorbachev, \emph{Method for solving the Delsarte problem for weighted designs on compact homogeneous spaces}, Chebyshevskii Sbornik \textbf{25} (2024), no.~4, 53--73.

\bibitem{Hamill}
T.\,M.~Hamill, J.\,S.~Whitaker, and C.~Snyder, \emph{Distance-dependent filtering of background error covariance estimates in an ensemble Kalman filter}, Mon. Wea. Rev. \textbf{129} (2001), 2776--2790.

\bibitem{husain} 
T.~Husain, \emph{Introduction to Topological Groups}, W.\,B.~Saunders Company, Philadelphia and London, 1966.

\bibitem{I-T}
A.\,D.~Ioffe and V.\,M.~Tikhomirov, \emph{Duality of convex functions and extremum problems},
Russian Math. Surveys \textbf{23} (6) (1968), 53--124.

\bibitem{KabLev}
G.\,A.~Kabatyanskii and V.\,I.~Levenshtein, 
\emph{On bounds for packing on the sphere and in space} (Russian), Probl. Inform. \textbf{14} (1978), no.~1, 3--25.

\bibitem{KLM} 
M.\,N.~Kolountzakis, N.~Lev, and M.~Matolcsi, \emph{The Tur\'an and Delsarte problems and their duals}, ArXiv preprint,  arXiv:2510.10172.

\bibitem{kolmatwei}
M.~N.~Kolountzakis, M.~Matolcsi, and M.~Weiner,
\emph{An application of positive definite functions to the problems of MUBs},
Proc.\ Amer.\ Math.\ Soc. \textbf{146} (2018), no.~3, 1143--1150.

\bibitem{kolrev}
M.\,N.~Kolountzakis and Sz.\,Gy.~R\'ev\'esz,
\emph{On a problem of Tur\'an about positive definite functions},
Proc. Amer. Math. Soc. \textbf{131} (2003), 3423--3430.

\bibitem{kuklin}
N.\,A.~Kuklin,
\emph{Delsarte method in the problem on kissing numbers in high-dimensional spaces},
Proc. Steklov Inst. Math. Suppl. \textbf{284} (2014), 108--123.

\bibitem{Levenshtein}
V.\,I.~Levenshtein, \emph{Bounds for packings in $n$-dimensional Euclidean space},
Dokl. Akad. Nauk SSSR \textbf{245} (1979) 1299--1303.

\bibitem{MatolcsiRuzsa}
M.~Matolcsi and I.\,Z.~Ruzsa, \emph{Difference Sets and Positive Exponential Sums I. General Properties},
J. Fourier Anal. Appl. \textbf{20} (2014), 17--41.

\bibitem{MW}
M. Matolcsi and M. Weiner,
\emph{An Improvement on the Delsarte-Type LP-Bound with Application to MUBs}, Open Syst. Inf. Dyn. \textbf{22} (2015), 1550001.

\bibitem{ramabulanaexistence}
M.\,D.~Ramabulana,
\emph{On the existence of an extremal function for the Delsarte extremal problem},
Anal. Math. \textbf{51} (2025), 279--291.

\bibitem{ramabulana}
M.\,D.~Ramabulana, \emph{Delsarte-type problems on homogeneous spaces and convolutions roots}, J. Fourier Anal. Appl. \textbf{32} (2026), 44.

\bibitem{revesz}
Sz.\,Gy.~Révész, \emph{Some trigonometric extremal problems and duality},
J. Aust. Math. Soc. Ser. A \textbf{50} (1991), 384--390.

\bibitem{Revesz-Beurling} 
Sz.\,Gy.~R\'ev\'esz, \emph{On Beurling's Prime Number Theorem}, Period. Math. Hung. {\bf 28} (1994) no. 3, 195--210.

\bibitem{R-Landau} 
Sz.\,Gy.~Révész, \emph{On some extremal problems of Landau}, Serdica Math. J. {\bf 33} (1) (2007) 125–162.

\bibitem{reveszLCA}
Sz.\,Gy.~R\'ev\'esz,
\emph{Tur{\'a}n’s extremal problem on locally compact Abelian groups},
Anal. Math. \textbf{37} (2009), 15--50.

\bibitem{Richmond}
T.~Richmond, \emph{General Topology}, De Gruyter, Berlin/Boston, 2020.

\bibitem{Ruzsa1981}
I.\,Z.~Ruzsa, \emph{Connections between the uniform distribution of a sequence and its differences},
in: Topics in Classical Number Theory,
Colloq. Math. Soc. János Bolyai, vol.~34,
North-Holland, Amsterdam--New York--Budapest, 1981, pp.~1419--1443.

\bibitem{Ruzsa1982} 
I.\,Z.~Ruzsa, 
\emph{Uniform distribution, positive trigonometric polynomials and difference sets.} Seminar on Number Theory, 1981/1982, Exp. No. 18, 18 pp
Université de Bordeaux I, U.E.R. de Mathématiques et d'Informatique, Laboratoire de Théorie des Nombres, Talence, 1982

\bibitem{sasvari}
Z.~Sasv\'{a}ri, \emph{Positive Definite and Definitizable Functions}, Vol. 2, Akademie Verlag, Berlin, 1994.

\bibitem{siegel} 
C.~L.~Siegel, \emph{{\"U}ber Gitterpunkte in konvexen K{\"o}rpern und ein damit zusammenh{\"a}ngendes Extremalproblem},
Acta Math. \textbf{65} (1935), 307--323.

\bibitem{viazovska} 
M.\,S.~Viazovska, \textit{The sphere packing problem in dimension 8}, Ann. of Math. (3) \textbf{185} (2017), 991--1015.

\bibitem{virosztek} 
D.~Virosztek, \emph{Applications of an intersection formula to dual cones}, Bull. Aust. Math. Soc. \textbf{97} (2018), 94--101.

\bibitem{wackenhuth}  
M.~Wackenhuth,
\emph{Linear programming bounds on homogeneous spaces, I: optimal packing density}, ArXiv preprint, arXiv:2505.23572.

\bibitem{maximilian}
M.~Wackenhuth,
\emph{Bounds on hyperbolic sphere packings: on a conjecture by Cohn and Zhao}, C. R., Math., Acad. Sci. Paris  
\textbf{364} (2026), 237--242.

\bibitem{wolf}
J.\,A.~Wolf, \emph{Harmonic Analysis on Commutative Spaces}, American Mathematical Society, Providence, 2007.

\bibitem{Yudin}
V.\,A.~Yudin, \emph{Packings of balls in Euclidean space, and extremal problems for trigonometric polynomials} (Russian), 
Diskret. Mat. \textbf{1} (1989) 155--158; translation in Discrete Math. Appl. \textbf{1} (1991) 69--72.


\end{thebibliography}

\end{document}